\documentclass[a4paper]{article}

\usepackage[utf8]{inputenc}

\usepackage{lmodern}
\usepackage{array}
\usepackage{graphicx}
\usepackage{subfig}
\usepackage{amssymb, amsfonts, amsmath, amsthm}
\usepackage{enumerate}

\usepackage[english]{babel}

\usepackage[colorlinks]{hyperref}

\usepackage{tikz}
\usetikzlibrary{decorations}

\usetikzlibrary{decorations.pathreplacing}

\tikzset{individu/.style={draw,thick}}

\numberwithin{equation}{section}

\theoremstyle{plain}
\newtheorem{theorem}{Theorem}[section]

\newtheorem{lemma}[theorem]{Lemma}
\newtheorem{proposition}[theorem]{Proposition}

\theoremstyle{definition}

\theoremstyle{remark}
\newtheorem{remark}[theorem]{Remark}

\newcommand{\N}{\mathbb{N}}
\newcommand{\Z}{\mathbb{Z}}
\newcommand{\R}{\mathbb{R}}

\newcommand{\calL}{\mathcal{L}}

\newcommand{\calA}{\mathcal{A}}

\newcommand{\calN}{\mathcal{N}}
\newcommand{\calM}{\mathcal{M}}

\newcommand{\ind}[1]{\mathbf{1}_{\left\{#1\right\}}}
\newcommand{\indset}[1]{\mathbf{1}_{#1}}
\newcommand{\floor}[1]{{\left\lfloor #1 \right\rfloor}}
\newcommand{\ceil}[1]{{\left\lceil #1 \right\rceil}}

\DeclareMathOperator{\E}{\mathbf{E}}
\renewcommand{\P}{\mathbf{P}}

\renewcommand{\bar}[1]{\overline{#1}}
\renewcommand{\tilde}[1]{\widetilde{#1}}
\renewcommand{\epsilon}{\varepsilon}
\renewcommand{\phi}{\varphi}

\newcommand{\Addresses}{{
  \bigskip
  \footnotesize

  Bastien Mallein, \textsc{LAGA - Institut Galil\'ee, 99 avenue Jean-Baptiste Cl\'ement
93430 Villetaneuse, France}\par\nopagebreak
  \textit{E-mail address}: \texttt{mallein@math.univ-paris13.fr}
  
  \medskip

  Sanjay Ramassamy, \textsc{Mathematics Department, Brown University, Box 1917, 151 Thayer street, Providence, RI 02912, USA}\par\nopagebreak
  \textit{E-mail address}: \texttt{sanjay.ramassamy@ipht.fr}

}}

\title{Barak-Erd\H{o}s graphs and the infinite-bin model}
\author{Bastien Mallein \and Sanjay Ramassamy}
\date{\today}

\newcommand{\egaldistr}{\overset{(d)}{=}}
\renewcommand{\hat}[1]{\widehat{#1}}

\begin{document}

\maketitle

\begin{abstract}
A Barak-Erd\H{o}s graph is a directed acyclic version of the Erd\H{o}s-R\'enyi random graph. It is obtained by performing independent bond percolation with parameter $p$ on the complete graph with vertices $\{1,...,n\}$, in which the edge between two vertices $i<j$ is directed from $i$ to $j$. The length of the longest path in this graph grows linearly with the number of vertices, at rate $C(p)$. In this article, we use a coupling between Barak-Erd\H{o}s graphs and infinite-bin models to provide explicit estimates on $C(p)$. More precisely, we prove that the front of an infinite-bin model grows at linear speed, and that this speed can be obtained as the sum of a series. Using these results, we prove the analyticity of $C$ for $p >1/2$, and compute its power series expansion. We also obtain the first two terms of the asymptotic expansion of $C$ as $p \to 0$, using a coupling with branching random walks.
\end{abstract}

\section{Introduction}
\label{sec:introduction}

Random graphs and interacting particle systems have been two active fields of research in probability in the past decades. In 2003, Foss and Konstantopoulos~\cite{FK} introduced a new interacting particle system called the infinite-bin model and established a correspondence between a certain class of infinite-bin models and Barak-Erd\H{o}s random graphs, which are a directed acyclic version of Erd\H{o}s-R\'enyi graphs.

In this article, we study the speed at which the front of an infinite-bin model drifts to infinity. These results are applied to obtain a fine asymptotic of the length of the longest path in a Barak-Erd\H{o}s graph. In the remainder of the introduction, we first describe Barak-Erd\H{o}s graphs, then infinite-bin models. We then state our main results on infinite-bin models, and their consequences for Barak-Erd\H{o}s graphs.

\subsection{Barak-Erd\texorpdfstring{\H{o}}{o}s graphs}

Barak and Erd\H{o}s introduced in~\cite{BE} the following model of a random directed graph with vertex set $\{1,\ldots,n\}$ (which we refer to as Barak-Erd\H{o}s graphs from now on) : for each pair of vertices $i<j$, add an edge directed from $i$ to $j$ with probability $p$, independently for each pair. They were interested in the maximal size of strongly independent sets in such graphs.

However, one of the most widely studied properties of Barak-Erd\H{o}s graphs has been the length of its longest path. It has applications to mathematical ecology (food chains)~\cite{CN,NC}, performance evaluation of computer systems (speed of parallel processes)~\cite{GNPT,IN} and queuing theory (stability of queues)~\cite{FK}.

Newman~\cite{N} studied the length of the longest path in Barak-Erd\H{o}s graphs in several settings, when the edge probability $p$ is constant (dense case), but also when it is of the form $c_n/n$ with $c_n=o(n)$ (sparse case). In the dense case, he proved that when $n$ gets large, the length of the longest path $L_n(p)$ grows linearly with $n$ in the first-order approximation :
\begin{equation}
  \label{eq:defC}
\lim_{n\rightarrow \infty}{\frac{L_n(p)}{n}}=C(p) \text{ a.s.},
\end{equation}
where the linear growth rate $C$ is a function of $p$. We plot in Figure~\ref{fig:cpgraph} an approximation of $C(p)$.

\begin{figure}[htbp]
\centering
\includegraphics[height=2in]{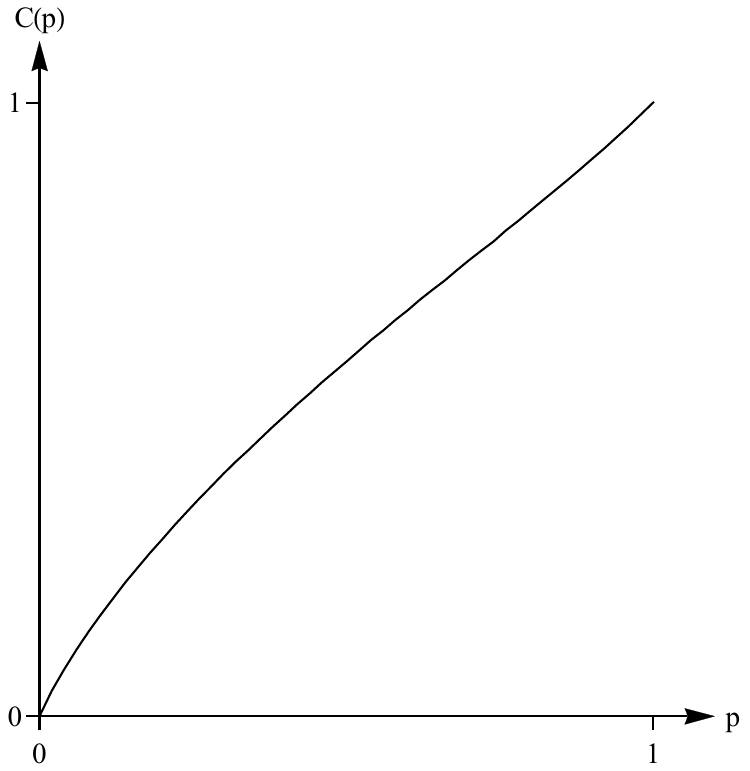}
\caption{Plot of a an approximation of $C(p)$, using $600,000$ iterations of an infinite-bin model, for values of $p$ that are integer multiples of $0.02$.}
\label{fig:cpgraph}
\end{figure}

Newman proved that the function $C$ is continuous and computed its derivative at $p=0$. Foss and Konstantopoulos~\cite{FK} studied Barak-Erd\H{o}s graphs under the name of ``stochastic ordered graphs'' and provided upper and lower bounds for $C$, obtaining in particular that
\begin{equation}
  \label{eqn:taylorFK}
C(1-q)=1-q+q^2-3q^3+7q^4+O(q^5) \text{ when } q\rightarrow0,
\end{equation}
where $q=1-p$ denotes the probability of the absence of an edge.

Denisov, Foss and Konstantopoulos~\cite{DFK} introduced the more general model of a directed slab graph and proved a law of large numbers and a central limit theorem for the length of its longest path. Konstantopoulos and Trinajsti\'c~\cite{KT} looked at a directed random graph with vertices in $\mathbb{Z}^2$ (instead of $\mathbb{Z}$ for the infinite version of Barak-Erd\H{o}s graphs) and identified fluctuations following the Tracy-Widom distribution. Foss, Martin and Schmidt~\cite{FMS} added to the original Barak-Erd\H{o}s model random edge lengths, in which case the problem of the longest path can be reformulated as a last-passage percolation question. Gelenbe, Nelson, Philips and Tantawi~\cite{GNPT} studied a similar problem, but with random weights on the vertices rather than on the edges.

Ajtai, Koml\'os and Szemer\'edi~\cite{AKS} studied the asymptotic behaviour of the longest path in sparse Erd\H{o}s-R\'enyi graphs, which are the undirected version of Barak-Erd\H{o}s graphs.

\subsection{The infinite-bin model}

Foss and Konstantopoulos introduced the infinite-bin model in~\cite{FK} as an interacting particle system which, for a right choice of parameters, gives information about the growth rate $C(p)$ of the longest path in Barak-Erd\H{o}s graphs. Consider a set of bins indexed by the set of integers $\mathbb{Z}$. Each bin may contain any number of balls, finite or infinite. A configuration of balls in bins is called \emph{admissible} if there exists $m \in \Z$ such that:
\begin{enumerate}
\item every bin with an index smaller or equal to $m$ is non-empty ;
\item every bin with an index strictly larger than $m$ is empty.
\end{enumerate} 
The largest index of a non-empty bin $m$ is called the position of the \emph{front}. From now on, all configurations will implicitly be assumed to be admissible. Given an integer $k\geq 1$, we define the \emph{move of type $k$} as a map $\Phi_k$ from the set of configurations to itself. Given an initial configuration $X$, $\Phi_k(X)$ is obtained by adding one ball to the bin of index $b_k+1$, where $b_k$ is the index of the bin containing the $k$-th ball of $X$ (the balls are counted from right to left, starting from the rightmost nonempty bin).

\begin{figure}[htbp]
\centering
\subfloat[A configuration $X$, the numbers inside the balls indicate how they are counted from right to left.]{\includegraphics[height=1.2in]{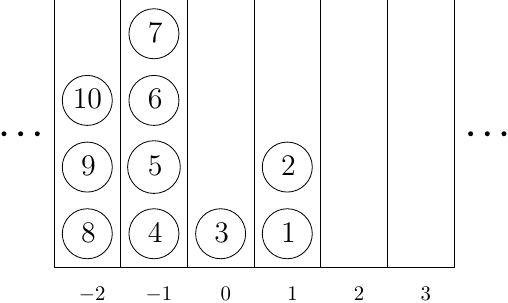}}

\hspace{\stretch{1}}
\subfloat[The configuration $\Phi_5(X)$.]{\includegraphics[height=1.2in]{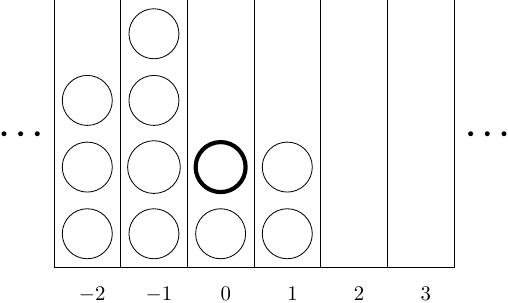}}
\hspace{\stretch{1}}
\subfloat[The configuration $\Phi_2(X)$.]{\includegraphics[height=1.2in]{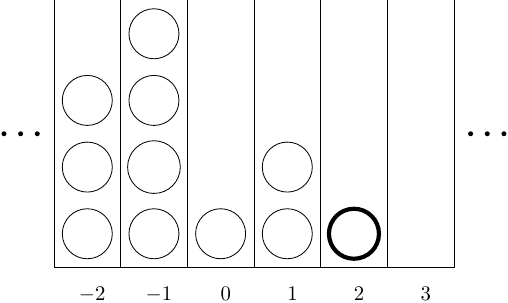}}\hspace{\stretch{1}}
\caption{Action of two moves on a configuration.}
\label{fig:move}
\end{figure}

Given a probability distribution $\mu$ on the set of positive integers and an initial configuration $X_0$, one defines the Markovian evolution of the infinite-bin model with distribution $\mu$ (or IBM($\mu$) for short) as the following stochastic recursive sequence:
\[
X_{n+1}=\Phi_{\xi_{n+1}}(X_n) \text{ for } n\geq0,
\]
where $(\xi_n)_{n\geq1}$ is an i.i.d. sequence of law $\mu$. We prove in Theorem~\ref{thm:existsSpeed} that the front moves to the right at a speed which tends a.s. to a constant limit~$v_{\mu}$. We call $v_{\mu}$ the \emph{speed} of the IBM($\mu$). Note that the model defined in~\cite{FK} was slightly more general, allowing $(\xi_n)_{n\geq1}$ to be a stationary-ergodic sequence. We also do not adopt their convention of shifting the indexing of the bins which forces the front to always be at position $0$.

Foss and Konstantopoulos~\cite{FK} proved that if $\mu_p$ is the geometric distribution of parameter $p$ then $v_{\mu_p}=C(p)$, where $C(p)$ is the growth rate of the length of the longest path in Barak-Erd\H{o}s graphs with edge probability $p$. They also proved, for distributions $\mu$ with finite mean verifying $\mu(\{1\})>0$, the existence of renovations events, which yields a functional law of large numbers and a central limit theorem for the IBM($\mu$). Based on a coupling result for the infinite-bin model obtained by Chernysh and Ramassamy~\cite{CR}, Foss and Zachary~\cite{FZ} managed to remove the condition $\mu(\{1\})>0$ required by~\cite{FK} to obtain renovation events.

Aldous and Pitman~\cite{AP} had already studied a special case of the infinite-bin model, namely what happens to the speed of the front when $\mu$ is the uniform distribution on $\{1,\ldots,n\}$, in the limit when $n$ goes to infinity. They were motivated by an application to the running time of local improvement algorithms defined by Tovey~\cite{To}.

\subsection{Speed of infinite-bin models}

The remainder of the introduction is devoted to the presentation of the main results proved in this paper. In this subsection we state the results related to general infinite-bin models, and in the next one we state the results related to the Barak-Erd\H{o}s graphs.

We first prove that in every infinite-bin model, the front moves at linear speed. Foss and Konstantopoulos~\cite{FK} had derived a special case of this result, when the distribution $\mu$ has finite expectation.
\begin{theorem}
\label{thm:existsSpeed}
Let $(X_n)$ be an infinite-bin model with distribution $\mu$, starting from an admissible configuration $X_0$. For any $n \in \N$, we write $M_n$ for the position of the front of $X_n$. There exists $v_\mu \in [0,1]$, depending only on the distribution $\mu$, such that
\[
  \lim_{n \to +\infty} \frac{M_n}{n} = v_\mu \quad \text{a.s.}
\]
\end{theorem}

In the next result, we obtain an explicit formula for the speed $v_\mu$ of the IBM($\mu$), as a series. To give this formula we first introduce some notation. Recalling that $\N$ is the set of positive integers, we denote by $\mathcal{A}$ the set of words on the alphabet $\N$, \textit{i.e.} the set of all finite-length sequences of elements of $\N$. Given a non-empty word $\alpha\in\mathcal{A}$, written $\alpha=(\alpha_1,\alpha_2,\ldots,\alpha_n)$ (where the $\alpha_i$ are the letters of $\alpha$), we denote by $L(\alpha) = n$ the length of $\alpha$. The empty word is denoted by $\emptyset$.

Fix an infinite-bin model configuration $X$. We define the subset $\mathcal{P}_X$ of $\mathcal{A}$ as follows: a word $\alpha$ belongs to $\mathcal{P}_X$ if it is non-empty, and if starting from the configuration $X$ and applying successively the moves $\Phi_{\alpha_1},\ldots,\Phi_{\alpha_n}$, the last move $\Phi_{\alpha_n}$ results in placing a ball in a previously empty bin.

Given a word $\alpha\in\mathcal{A}$ which is not the empty word, we set $\varpi \alpha \in\mathcal{A}$ to be the word obtained from $\alpha$ by removing the first letter. We also set $\varpi \emptyset=\emptyset$. We define the function $\epsilon_X:\mathcal{A}\rightarrow\{-1,0,1\}$ as follows:
\[
  \epsilon_X(\alpha)= \ind{\alpha \in \mathcal{P}_X}-\ind{\varpi \alpha \in \mathcal{P}_X}.
\]

\begin{theorem}
\label{thm:seriesformula}
Let $X$ be an admissible configuration and $\mu$ a probability distribution on $\N$. We define the weight of a word $\alpha=(\alpha_1,\ldots,\alpha_n)$ by
\[
  W_{\mu}(\alpha)=\prod_{i=1}^n \mu\left(\left\{\alpha_i\right\}\right) = \P(\alpha = (\xi_1,\ldots \xi_n)).
\]
If $\sum_{\alpha\in\mathcal{A}} |\epsilon_X(\alpha)|W_\mu(\alpha) < +\infty$, then
\begin{equation}
\label{eq:generalseriesformula}
  v_\mu =\sum_{\alpha\in\mathcal{A}} \epsilon_X(\alpha)W_\mu(\alpha).
\end{equation}
\end{theorem}

\begin{remark}
\label{rem:dependencyX}
One of the most striking features of \eqref{eq:generalseriesformula} is that whereas for any $\alpha \in \mathcal{A}$, $X \mapsto \epsilon_X(\alpha)$ is a non-constant function of $X$, $v_\mu$ does not depend of this choice of configuration. As a result, Theorem~\ref{thm:seriesformula} gives in fact an infinite number of formulas for the speed $v_\mu$ of the IBM($\mu$).
\end{remark}

Theorem~\ref{thm:seriesformula} can be extended to prove the following result:
\[
  v_\mu = \lim_{n \to +\infty} \frac{1}{n}\sum_{k = 1}^n \sum_{\alpha \in \mathcal{A} : L(\alpha) \leq k} \epsilon_X(\alpha) W_\mu(\alpha).
\]
In other words, if we define $\sum_{\alpha\in\mathcal{A}} \epsilon_X(\alpha)W_\mu(\alpha)$ as the Ces\`aro mean of its partial sums (on words of finite length), \eqref{eq:generalseriesformula} holds for any probability distribution $\mu$ and admissible configuration $X$.

\subsection{Longest increasing paths in Barak-Erd\texorpdfstring{\H{o}}{o}s graphs}

Using the coupling introduced by Foss and Konstantopoulos between Barak-Erd\H{o}s graphs and infinite-bin models, we use the previous results to extract information on the function $C$ defined in \eqref{eq:defC}. Firstly, we prove that for $p$ large enough (i.e. for dense Barak-Erd\H{o}s graphs), the function $C$ is analytic and we obtain the power series expansion of $C(p)$ centered at~$1$. Secondly, we provide the first two terms of the asymptotic expansion of $C(p)$ as $p\rightarrow0$.

We deduce from Theorem~\ref{thm:seriesformula} the analyticity of $C(p)$ for $p$ close to $1$. For any word $\alpha \in \mathcal{A}$, we define the height of $\alpha$ to be
\[
H(\alpha)=\sum_{i=1}^{L(\alpha)}{\alpha_i}-L(\alpha).
\]
For any $k \in \N$ and admissible configuration $X$, we set
\begin{equation}
  \label{eqn:formulaCoeff}
  a_k = \sum_{\alpha \in \mathcal{A}: H(\alpha) \leq k, L(\alpha) \leq k+1} \epsilon_X(\alpha) (-1)^{k-H(\alpha)} \binom{L(\alpha)}{k-H(\alpha)}.
\end{equation}

\begin{theorem}
\label{thm:powerseries}
The function $C$ is analytic on $\left( \frac{1}{2}, 1 \right]$ and for $p \in \left( \frac{3-\sqrt{2}}{2}, 1 \right]$,
\[C(p) = \sum_{k\geq0} a_k (1-p)^k.\]
\end{theorem}

Similarly to what has been observed in Remark~\ref{rem:dependencyX}, this result proves that the value of $a_k$ does not depend on the configuration $X$, justifying \textit{a posteriori} the notation. In a recent work \cite{MR} accomplished after the present article was completed, we show that $C(p)$ is actually analytic on $(0,1]$, so the bound $\frac{1}{2}$ in the above theorem is not optimal. Similarly, we do not expect the bound $\frac{3 - \sqrt{2}}{2}$ for the radius of convergence of the Taylor expansion at $1$ to be optimal. Numerical simulations tend to suggest that the power series expansion of $C(p)$ at $p=1$ has a radius of convergence between $0.5$ and $1$.

\begin{remark}
\label{rem:numerical}
Using \eqref{eqn:formulaCoeff} and Lemma~\ref{lem:lengthsmallerheight}, it is possible to explicitly compute as many coefficients of the power series expansion as desired, by picking a configuration $X$ and computing quantities of the form $\epsilon_X(\alpha)$ for finitely many words $\alpha\in\mathcal{A}$. For example, we observe that as $q\rightarrow0$,
\[
C(1-q)=1-q+q^2-3q^3+7q^4-15q^5+29q^6-54q^7+102q^8+O(q^9).
\]
It is clear from formula \eqref{eqn:formulaCoeff} that $(a_k)$ is integer-valued. Based on our computations, we conjecture that $((-1)^ka_k, k \geq 0)$ is non-negative and non-decreasing.
\end{remark}

We now turn to the asymptotic behaviour of $C(p)$ as $p \to 0$, i.e. the length of the longest increasing path in sparse Barak-Erd\H{o}s graphs. We precise the asymptotic estimate obtained by Newman~\cite{N}, namely that $C(p) \sim e p$ as $p \to 0$.
\begin{theorem}
\label{thm:smallp}
We have $\displaystyle C(p) = e p - \frac{\pi^2 e}{2}p (-\log p)^{-2} + o(p (-\log p)^{-2})$.
\end{theorem}
In particular, this result proves that the function $C(p)$ has no finite second derivative at point $p=0$. 


Theorem~\ref{thm:smallp} is obtained by coupling the infinite-bin model with uniform distribution with a continuous-time branching random walk with selection (as observed by Aldous and Pitman \cite{AP}) and by extending to the continuous-time setting the results of B\'erard and Gou\'er\'e \cite{BeG10} on the asymptotic behaviour of a discrete-time branching random walk. Assuming that the conjecture of Brunet and Derrida~\cite{BD97} on the speed of a branching random walk with selection holds, and that the coupling of Aldous and Pitman is precise enough for the asymptotic expansion to be transferred to the infinite-bin model setting, the next term in the asymptotic expansion should be given by $3e\pi^2 p \tfrac{\log(-\log p)}{(-\log p)^3}$.

\begin{remark}
\label{rem:sparserGraphs}
With arguments similar to the ones used to prove Theorem~\ref{thm:smallp}, we expect that one can also obtain the asymptotic behaviour of $L_n(p)$ as $n \to +\infty$ and $p \to 0$ simultaneously, proving that:
\[
  L_n(p_n) = n e p_n - n\frac{\pi^2e }{2} p_n (-\log p_n)^2 + o(n p_n (\log p_n)^{-2}) \text{ in probability},
\]
as long as $p_n \gg \frac{(\log n)^3}{n}$. We expect a different behaviour if $p_n \sim \lambda \frac{(\log n)^3}{n}$. We mention that Itoh \cite{Itoh} studied the asymptotic behaviour of $L_n(a/n)$ as $n \to \infty$.
\end{remark}

\subsection*{Organisation of the paper}
We state more precisely the notation used to study the infinite-bin model in Section~\ref{sec:generalities}. We also introduce an increasing coupling between infinite-bin models, which is a key result for the rest of the article.

In Section~\ref{sec:finiteSupport}, we prove that the speed of an infinite-bin model with a measure of finite support can be expressed using the invariant measure of a finite Markov chain. This result is then used to prove Theorem~\ref{thm:existsSpeed} in the general case. We prove Theorem~\ref{thm:seriesformula} in Section~\ref{sec:seriesFormula} using a method akin to ``exact perturbative expansion''.

We review in Section~\ref{sec:geo} the Foss-Konstantopoulos coupling between Barak-Erd\H{o}s graphs and the infinite-bin model and use it to provide a sequence of upper and lower bounds converging exponentially fast to $C(p)$. This coupling is used in Section~\ref{sec:powerseries}, where we prove Theorem~\ref{thm:powerseries} using Theorem~\ref{thm:seriesformula}. Finally, we prove Theorem~\ref{thm:smallp} in Section~\ref{sec:smallp}, by extending the results of B\'erard and Gou\'er\'e \cite{BeG10} to compute the asymptotic behaviour of a continuous-time branching random walk with selection.

\section{Basic properties of the infinite-bin model}
\label{sec:generalities}

We write $\N$ for the set of positive integers, $\bar{\N} = \N \cup \{ +\infty\}$, $\Z_+$ for the set of non-negative integers and $\bar{\Z}_+ = \Z_+ \cup \{ +\infty\}$. We denote by
\[
  S = \left\{ X \in (\bar{\Z}_+)^\Z : \begin{array}{l}
    \exists m \in \Z : \forall j \in \Z,  X(j) = 0 \iff j > m \quad \text{and}\\
    \forall j \in \Z, X(j) = +\infty \Rightarrow X(j-1) = +\infty
  \end{array}\right\}
\]
the set of admissible configurations for an infinite-bin model. Note that the definition we use here is more restrictive than the one used, as a simplification, in the introduction. Indeed, we impose here that if a bin has an infinite number of balls, every bin to its left also has an infinite number of balls. However, this has no impact on our results, as the dynamics of an infinite-bin model does not affect bins to the left of a bin with an infinite number of balls. One does not create balls in a bin at distance greater than 1 from a non-empty bin.

We wish to point out that our definition of admissible configurations has been chosen out of convenience. Most of the results of this article could easily be generalized to infinite-bin models with a starting configuration belonging to
\[
  S^0 = \left\{ X \in (\bar{\Z}_+)^\Z : \lim_{k \to +\infty} X(k) = 0 \text{ and } \sum_{k \in \Z} X(k) = +\infty\right\},
\]
see e.g. Remark~\ref{rem:generalized}. They could even be generalized to configurations starting with a finite number of balls, if we adapt the dynamics of the infinite-bin model as follows. For any $n \in \N$, if $\xi_n$ is larger than the number of balls existing at time $n$, then the step is ignored and the IBM configuration is not modified. However, with this definition some trivial cases might arise, for example starting with a configuration with only one ball, and using a measure $\mu$ with $\mu(\{1\})=0$.

For any $X \in S$ and $k \in \Z$, we call $X(k)$ the number of balls at position $k$ in the configuration $X$. Observe that the set of non-empty bins is a semi-infinite interval of $\Z$. In particular, for any $X \in S$, there exists a unique integer $m \in \Z$ such that $X(m) \neq 0$ and $X(j)=0$ for all $j >m$. The integer $m$ is called the \emph{front} of the configuration.

Let $X \in S$, $k \in \Z$ and $\xi \in \N$. We denote by 
\[N(X,k) = \sum_{j=k}^{+\infty} X(j) \quad \text{and} \quad B(X,\xi) = \inf\{j \in \Z : N(X,j) < \xi \}\]
the number of balls to the right of $k$ and the leftmost position such that there are less than $\xi$ balls to its right respectively. Note that the position of the front in the configuration $X$ is given by $B(X,1)-1$. Observe that for any $X \in S$,
\begin{equation}
  \label{eqn:lips}
  \forall 1 \leq \xi \leq \xi', \; 0 \leq B(X,\xi)-B(X,\xi') \leq  \xi'-\xi.
\end{equation}

For $\xi \in \N$ and $X \in S$, we set $\Phi_\xi(X) = \left( X(j) + \ind{j = B(X,\xi)}, j \in \Z \right)$ the transformation that adds one ball to the right of the $\xi$-th rightmost ball in $X$. We extend the notation to allow $\xi \in \bar{\N}$, by setting $\Phi_\infty(X)=X$. We also introduce the shift operator $\tau(X) = \left( X(j-1), j \in \Z \right)$. We observe that $\tau$ and~$\Phi_\xi$ commute, i.e.
\begin{equation}
  \label{eqn:commutates}
  \forall X \in S, \forall \xi \in \bar{\N}, \Phi_\xi(\tau(X)) = \tau(\Phi_\xi(X)).
\end{equation}

Recall that an infinite-bin model consists in the sequential application of randomly chosen transformations $\Phi_\xi$, called move of type $\xi$. More precisely, given $\mu$ a probability measure on $\bar{\N}$ and $(\xi_n, n \geq 1)$ i.i.d. random variables with distribution $\mu$, the IBM($\mu$) $(X_n)$ is the Markov process on $S$ starting from $X_0 \in S$, such that for any $n \geq 0$, $X_{n+1} = \Phi_{\xi_{n+1}}(X_n)$.

We introduce a partial order on $S$, which is compatible with the infinite-bin model dynamics: for any $X,Y\in S$, we write
\[
  X \preccurlyeq Y \iff \forall j \in \Z, N(X,j) \leq N(Y,j) \iff \forall \xi \in \N, B(X,\xi) \leq B(Y,\xi).
\]
The functions $(\Phi_\xi)$ are monotone, increasing in $X$ and decreasing in $\xi$ for this partial order. More precisely
\begin{equation}
  \label{eqn:PhiMonotone}
  \forall X \preccurlyeq Y\in S, \; \forall 1 \leq \xi \leq \xi' \leq \infty, \; \Phi_{\xi'}(X) \preccurlyeq \Phi_\xi(Y).
\end{equation}
Moreover, the shift operator $\tau$ dominates every function $\Phi_\xi$, i.e.
\begin{equation}
  \label{eqn:PhiMonotone2}
  \forall X \preccurlyeq Y\in S, \; \forall 1 \leq \xi \leq \infty, \; \Phi_{\xi}(X) \preccurlyeq \tau(Y).
\end{equation}
As a consequence, infinite-bin models can be coupled in an increasing fashion.

\begin{proposition}
\label{prop:increasingCoupling}
Let $\mu$ and $\nu$ be two probability distributions on $\bar{\N}$, and $X_0 \preccurlyeq Y_0 \in S^0$. If $\mu([1,k]) \leq \nu([1,k])$ for any $k \in \N$, we can couple the IBM$(\mu)$ $(X_n)$ and the IBM$(\nu)$ $(Y_n)$ such that for any $n \geq 0$, $X_n \preccurlyeq Y_n$ a.s.
\end{proposition}

\begin{proof}
As for any $k \in \N$, $\mu([1,k]) \leq \nu([1,k])$, we can construct a pair $(\xi,\zeta)$ such that $\xi$ has law $\mu$, $\zeta$ has law $\nu$ and $\xi \geq \zeta$ a.s. Let $(\xi_n,\zeta_n)$ be i.i.d. copies of $(\xi,\zeta)$, we set $X_{n+1} = \Phi_{\xi_{n+1}} (X_n)$ and $Y_{n+1} = \Phi_{\zeta_{n+1}}(Y_n)$. By induction, using~\eqref{eqn:PhiMonotone}, we immediately have $X_n \preccurlyeq Y_n$ for any $n \geq 0$.
\end{proof}

We extended in this section the definition of the IBM($\mu$) to measures with positive mass on $\{\infty\}$. As applying $\Phi_\infty$ does not modify the ball configuration, the IBM($\mu$) and the IBM($\mu(.|.<\infty)$) are straightforwardly connected.
\begin{lemma}
\label{lem:obvious}
Let $\mu$ be a probability measure on $\bar{\N}$ with $p := \mu(\{\infty\})<1$. We write $\nu$ for the measure verifying $\nu(\{k\}) = \frac{\mu(\{k\})}{1-p}$ for all $k \in \N$. Let $(X_n)$ be an IBM($\nu$)  and $(S_n)$ be an independent random walk with step distribution Bernoulli with parameter $1-p$. Then the process $(X_{S_n}, n \geq 0)$ is an IBM($\mu$).
\end{lemma}

In particular, assuming Theorem~\ref{thm:existsSpeed} holds, we would have $v_\mu = (1-p) v_\nu$ with the notation of the previous lemma.

\section{Speed of the infinite-bin model}
\label{sec:finiteSupport}

In this section, we prove the existence of a well-defined notion of speed of the front of an infinite-bin model. We first discuss the case when the distribution $\mu$ is finitely supported and the initial configuration is simple, then we extend it to any distribution $\mu$ and finally we generalize to any admissible initial configuration.

\subsection{Infinite-bin models with finite support}
\label{subsec:finite}

Let $\mu$ be a probability measure on $\bar\N$ with finite support, i.e. such that there exists $K \in \N$ verifying $\mu([K+1,+\infty))=0$. Let $(X_n)$ be an IBM($\mu$), we say that $(X_n)$ is an infinite-bin model with support bounded by $K$. One of the main observations of this subsection is that such an infinite-bin model can be studied using a Markov chain on a finite set. As a consequence, we obtain an expression for the speed of this infinite-bin model.

Given $K \in \N$, we introduce the set
\[
  S_K = \left\{ x \in \Z_+^{K-1} : \sum_{i=1}^{K-1} x_i < K \text{ and } \forall 1 \leq i \leq j \leq K-1, x_i=0 \Rightarrow x_j = 0 \right\}.
\]
For any $Y \in S_K$, we write $|Y| = \sum_{j=1}^{K-1} Y(j)$. We introduce
\[
  \Pi_K :\begin{array}{rcl}
  S & \longrightarrow & S_K\\
  X & \longmapsto & \left( X(B(X,K)+j-1), 1 \leq j \leq K-1 \right).
\end{array}
\]
For any $n \in \N$, we write $Y_n = \Pi_K(X_n)$, that encodes the set of balls that are close to the front. As the IBM has support bounded by $K$, the bin in which the $(n+1)$-st ball is added to $X_n$ depends only on the position of the front and on the value of $Y_n$. This reduces the study of the dynamics of $(X_n)$ to the study of $(Y_n, n \geq 1)$.
\begin{lemma}
\label{lem:markov}
The sequence $(Y_n)$ is a Markov chain on $S_K$ with a unique stationary probability distribution.
\end{lemma}

\begin{proof}
For any $1 \leq \xi \leq K$ and $Y \in S_K$, we denote by
\[
  \tilde{B}(Y,\xi) = \begin{cases}
    \min\{ k \geq 1 : \sum_{i=k}^{K-1} Y(i) < \xi \} & \mathrm{if} \quad |Y| \geq \xi\\
    1 & \mathrm{otherwise,}
  \end{cases}
\]
\[
 \tilde{\Phi}_\xi(Y) = \begin{cases}
    \left( Y(j) + \ind{j=\tilde{B}(Y,\xi)}, 1\leq j \leq K-1 \right) & \mathrm{if} \quad |Y| < K-1\\
    \left( Y(j+1) + \ind{j+1=\tilde{B}(Y,\xi)}, 1\leq j \leq K-2, 0 \right) & \mathrm{if} \quad |Y| = K-1.
  \end{cases}
\]
For any $X \in S$ and $\xi \leq K$, we have $B(X,\xi)=B(X,K) + \tilde{B}(\Pi_K(X),\xi)-1$. Moreover, we have $\Pi_K\left(\Phi_\xi(X)\right) = \tilde{\Phi}_\xi(\Pi_K(X))$.

\begin{figure}[htbp]
\centering
\includegraphics[height=2in]{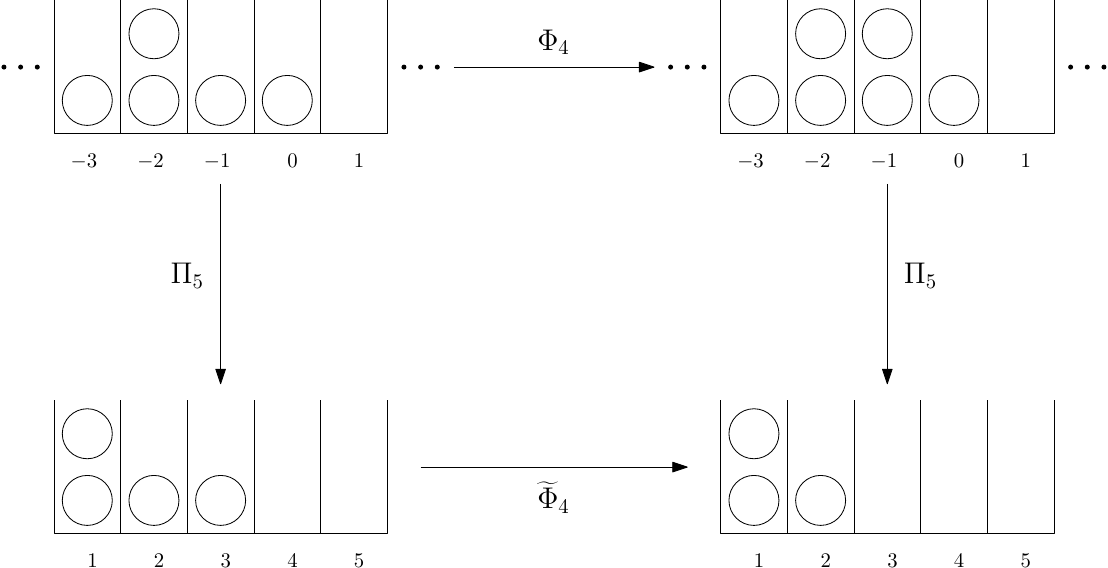}
\caption{``Commutation'' of $\Pi_5$ with $\Phi_4$ and $\tilde{\Phi}_4$.}
\label{fig:commutation}
\end{figure}

Let $(\xi_n)$ be i.i.d. random variables with law $\mu$ and $X_0 \in S$. For any $n \in \N$, we set $X_{n+1} = \Phi_{\xi_{n+1}}(X_n)$. Using the above observation, we have
\[
  Y_{n+1} = \Pi_K(X_{n+1}) = \Pi_K(\Phi_{\xi_{n+1}}(X_n))=\tilde{\Phi}_{\xi_{n+1}}(\Pi_K(X_n)) = \tilde{\Phi}_{\xi_{n+1}}(Y_n),
\]
thus $(Y_n)$ is a Markov chain.

Denote by $k$ the smallest integer in the support of $\mu$ and set $q:=\floor{\frac{K}{k}}$. One easily observes that, starting the chain $(Y_n)$ at an arbitrary state and applying moves of type $k$ sufficiently many times, one reaches the state with $q$ bins containing $k$ balls each. This entails that the finite state-space Markov chain $(Y_n)$ has a unique essential communicating class, hence it has a unique stationary probability distribution (see e.g. \cite[Proposition 1.26]{LPW}).
\end{proof}

For any $n \in \N$, the set of bins that are part of $Y_n$ represents the set of ``active'' bins in $X_n$, i.e. the bins in which a ball can be added at some time in the future with positive probability. The number of balls in $(Y_n)$ increases by one at each time step, until it reaches $K-1$. At this time, when a new ball is added, the leftmost bin ``freezes'', it will no longer be possible to add balls to this bin, and the ``focus'' is moved one step to the right.

We introduce a sequence of stopping times defined by
\[
  T_0=0 \quad \text{and} \quad T_{p+1} = \inf\{ n > T_p : |Y_{n-1}|=K-1 \}.
\]
We also set $Z_p = K-|Y_{T_p}|$ the number of balls in the bin that ``freezes '' at time $T_p$. For any $n \in \N$, we write $\tau_n =p$ for any $T_p \leq n < T_{p+1}$.
\begin{lemma}
\label{lem:linksXY}
Let $X_0 \in S$ such that $B(X_0,K)=1$, then
\begin{itemize}
  \item for any $p \geq 0$, $X_\infty(p) = Z_p$,
  \item for any $n \geq 0$ and $\xi \leq K$, $B(X_n,\xi) = \tau_n + B(Y_n,\xi)$.
\end{itemize}
\end{lemma}

\begin{proof}
By induction, for any $p \geq 0$, $B(X_{T_p},K) = p+1$. Consequently, for any $n \geq T_p$, we have $X_n(p) = X_{T_p}(p) = K - |Y_{T_p}| = Z_p$. Moreover, as
\[B(X_n,K) = \tau_n+1 \quad \text{and} \quad B(X_n,\xi) = B(X_n,K) + B(Y_n,\xi)-1,\]
we have the second equality.
\end{proof}

Using the above result, we prove that the speed of an infinite-bin model with finite support does not depend on the initial configuration. We also obtain a formula for the speed $v_\mu$, that can be used to compute explicit bounds.
\begin{proposition}
\label{prop:speedfinite}
Let $\mu$ be a probability measure with finite support and $X$ be an IBM($\mu$) with initial configuration $X_0\in S$. There exists $v_\mu \in [0,1]$ such that for any $\xi \in \N$, we have
\[\lim_{n \to +\infty} \frac{B(X_n,\xi)}{n} = v_\mu \quad \text{a.s.}\]
Moreover, setting $\pi$ for the invariant probability measure of $(Y_n)$ we have
\begin{equation}
  \label{eqn:equationSpeed}
  v_\mu = \frac{1}{\E_\pi(T_2-T_1)} = \frac{1}{\E_\pi(Z_1)}.
\end{equation}
\end{proposition}

\begin{proof}
Let $X_0 \in S$, we can assume that $B(X_0,K)=1$, up to a deterministic shift. At each time $n$, a ball is added in a bin with a positive index, thus for any $n \in \N$, we have
\[
  \sum_{j=1}^{+\infty} X_n(j) = n + \sum_{j=1}^{+\infty} X_0(j).
\]
Using the notation of Lemma~\ref{lem:linksXY}, we rewrite it $\sum_{j=1}^{\tau_n} Z_j + |Y_n| = n + \sum_{j=1}^{+\infty} X_0(j)$. Moreover, as $0 \leq |Y_n| \leq K$ and $0 \leq \sum_{j=1}^{+\infty} X_0(j) \leq K$, we have
\[
  1-\frac{K}{n}\leq \frac{\sum_{j=1}^{\tau_n} Z_j}{n} \leq 1+\frac{K}{n},
\]
yielding $\lim_{n \to +\infty} \frac{\sum_{j=1}^{\tau_n} Z_j}{n} = 1$ a.s. As $\lim_{p \to +\infty} T_p = +\infty$ a.s., we obtain
\[
  \lim_{p \to +\infty} \frac{\sum_{j=1}^{p} Z_j}{T_p} = 1 \quad \text{a.s.}
\]
Moreover $\lim_{p \to +\infty} \frac{1}{p} \sum_{j=1}^{p} Z_j = \E_\pi(Z_1)$ and $\lim_{p \to +\infty} \frac{T_p}{p} = \E_\pi(T_2-T_1)$ by ergodicity of $(Y_n)$. Consequently, if we set $v_\mu := \frac{1}{\E_\pi(T_2-T_1)} = \frac{1}{\E_\pi(Z_1)},$ the constant $v_\mu$ is well-defined.

We apply Lemma~\ref{lem:linksXY}, we have
\[
  \frac{B(X_n,1)}{n} = \frac{\tau_n}{n} + \frac{B(Y_n,1)}{n} \in \left[ \frac{\tau_n}{n}, \frac{\tau_n}{n} + \frac{K}{n} \right].
\]
Moreover, we have $\lim_{n \to +\infty} \frac{\tau_n}{n} = \lim_{p \to +\infty} \frac{p}{T_p} = v_\mu$ a.s. This yields
\begin{equation}
  \label{eqn:speedTemp}
  \lim_{n \to +\infty} \frac{B(X_n,1)}{n} = v_\mu \quad \text{a.s.}
\end{equation}
Using \eqref{eqn:lips}, this convergence is extended to $\lim_{n \to +\infty} \frac{B(X_n,\xi)}{n} = v_\mu$ a.s.
\end{proof}

\begin{remark}
\label{rem:finitesupportlazy}
If the support of $\mu$ is included in $\left[1,K\right]\cup\left\{+\infty\right\}$, it follows from Lemma~\ref{lem:obvious} that the IBM($\mu$) also has a well-defined speed $v_{\mu}$.
\end{remark}

\subsection{Extension to arbitrary distributions}
\label{subsec:arbitrary}

We now use Proposition~\ref{prop:speedfinite} to prove Theorem~\ref{thm:existsSpeed}.

\begin{proposition}
\label{prop:speedIBM}
Let $\mu$ be probability measure on $\N$ and $(X_n)$ an IBM($\mu$) with initial configuration $X_0\in S$. There exists $v_\mu \in [0,1]$ such that for any $\xi \in \N$, we have $\lim_{n \to +\infty} \frac{B(X_n,\xi)}{n} = v_\mu$ a.s.

Moreover, if $\nu$ is another probability measure we have
\begin{equation}
\label{eqn:encadrementSpeed}
 \forall k \in \N, \nu([1,k]) \leq \mu([1,k]) \Rightarrow v_\nu \leq v_\mu.
\end{equation}
\end{proposition}

\begin{proof}
Let $X_0 \in S$. We write $(\xi_n, n \geq 1)$ for an i.i.d. sequence of random variables of law $\mu$. For any $n,K \geq 1$, we set $\xi^K_n = \xi_n \ind{\xi_n \leq K} + \infty\ind{\xi_n > K}$. We then define the processes $(\underline{X}^K_n)$ and $(\bar{X}^K_n)$ by $\underline{X}^K_0 = \bar{X}^K_0 = X_0$ and
\[
  \underline{X}^K_{n+1} = \Phi_{\xi^K_{n+1}}(\underline{X}^K_n) \quad \text{and} \quad \bar{X}^K_{n+1} = \begin{cases}
   \Phi_{\xi_{n+1}}(\bar{X}^K_n) & \text{if } \xi_{n+1} \leq K\\
   \tau(\bar{X}^K_n) & \mathrm{otherwise.} 
  \end{cases}
\]
By induction, we have $\underline{X}^K_n \preccurlyeq X_n \preccurlyeq \bar{X}^K_n$ for any $n \geq 0$, using \eqref{eqn:PhiMonotone} and \eqref{eqn:PhiMonotone2}.

As $(\underline{X}^K_n)$ is an infinite-bin model with support included in $\left[1,K\right]\cup\left\{+\infty\right\}$, by Remark~\ref{rem:finitesupportlazy}, there exists $v_K \in [0,1]$ such that for any $\xi \in \N$
\[
  \liminf_{n \to +\infty} \frac{B(X_n,\xi)}{n} \geq \lim_{n \to +\infty} \frac{B(\underline{X}^K_n,\xi)}{n} = v_K \quad \text{a.s.}
\]
Moreover, by definition of $(\bar{X}^K_n)$ and \eqref{eqn:commutates}, for any $\xi,n \geq 1$ we have
\[
  B(\bar{X}^K_n,\xi) = B(\underline{X}^K_n,\xi) + \sum_{j=1}^n \ind{K< \xi_j < +\infty},
\]
therefore, by the law of large numbers
\[
  \limsup_{n \to +\infty} \frac{B(X_n,\xi)}{n} \leq \lim_{n \to +\infty} \frac{B(\bar{X}^K_n,\xi)}{n} = v_K + \mu([K+1,+\infty)) \quad \text{a.s.}
\]

By Proposition~\ref{prop:increasingCoupling}, we observe immediately that $(v_K)$ is an increasing sequence, bounded by 1, thus converges. Moreover, $\lim_{K \to +\infty} \mu([K+1,+\infty)) = 0$. We conclude that $\displaystyle \lim_{n \to +\infty} \tfrac{1}{n}B(X_n,\xi) = \lim_{K \to +\infty} v_K =:v_\mu$ a.s. By Proposition~\ref{prop:increasingCoupling}, \eqref{eqn:encadrementSpeed} trivially holds.
\end{proof}

\begin{remark}
\label{rem:encadrementSpeed}
Let $\mu$ be a probability measure on $\N$, we set $\mu_K = \mu(.|. \leq K)$. We observe from the proof of Proposition~\ref{prop:speedIBM} and Lemma~\ref{lem:obvious} that\[
  \mu([1,K]) v_{\mu_K} \leq v_\mu \leq \mu([1,K]) v_{\mu_K} + \mu([K+1,+\infty)).
\]
As $v_{\mu_K}$ is the speed of an IBM with support bounded by $K$, it can be computed explicitly using \eqref{eqn:equationSpeed}. This provides tractable bounds for $v_\mu$. For example, we have $v_\mu \geq \frac{\mu(\left\{K_0\right\})}{K_0}$, where $K_0 = \inf\{ k > 0 : \mu(k)>0\}$.
\end{remark}

\begin{remark}
\label{rem:generalized}
Proposition~\ref{prop:speedIBM} can be extended to infinite-bin models starting with a configuration $X \in S^0$. Let $\mu$ be a probability measure and $(X_n)$ an IBM($\mu$) starting with a configuration $X \in S^0$. If $\mu$ has a support bounded by $K$, then the projection $(\Pi_K(X_n))$ is a Markov chain that will hit the set $S_K$ in finite time. Therefore, we can apply Proposition~\ref{prop:speedfinite} and we have $\lim_{n \to +\infty} \frac{1}{n}B(X_n,1) = v_\mu$ a.s.

If $\mu$ has unbounded support, the IBM($\mu$) can still be bounded, in the same way as in the proof of Proposition~\ref{prop:speedIBM}, by infinite-bin models with bounded support. As a consequence, Theorem~\ref{thm:existsSpeed} holds for any starting configuration belonging to $S^0$.
\end{remark}

\section{A formula for the speed of the infinite-bin model}
\label{sec:seriesFormula}

In this section, we prove that we can write $v_\mu$ as the sum of a series, provided that this series converges. A non-rigorous heuristic for the proof goes along the following lines. Let $\eta>0$ and $\mu$ be a probability measure such that $\mu(\{1\}) \geq 1-\eta$, and $(\xi_n: n \in \N)$ be i.i.d. random variables with law $\mu$. If $\eta$ is small enough then the sequence $(\xi_n)$ consists in long time intervals such that $\xi_n = 1$ on these intervals, separated by short patterns that appear at random. Every move of type 1 makes the front of the infinite-bin model increase by 1, and each pattern induces a delay. Therefore, we expect the value of $v_{\mu}$ to be close to $1$ minus the sum over every possible pattern of the delay caused by this pattern to the process multiplied by its probability of occurrence.

This sum is an infinite sum and we hope that for $\eta$ small enough, the contributions of the long patterns will decay fast enough so that the series converges and its sum is equal to $v_\mu$. It appears that in fact, this series often converges, even when $\mu(1)$ is not close to 1, and when it converges its sum is equal to $v_\mu$.

We recall some notation from the introduction. We denote by $\calA$ the set of finite words on the alphabet $\N$. For any $\alpha = (\alpha_1,\ldots,\alpha_n) \in \calA$, we define
$L(\alpha) :=n$ to be the length of $\alpha$. 

Let $\mu$ be a probability distribution on $\N$ and $(\xi_j)_{j\geq1}$ be i.i.d. random variables with law $\mu$. We write
\[
  W_\mu(\alpha) := \prod_{j=1}^{L(\alpha)} \mu(\{\alpha_j\}) = \P((\xi_1,\ldots \xi_{L(\alpha)}) = \alpha)
\]
for the weight of the word $\alpha$.

If $\alpha=(\alpha_1,\ldots,\alpha_n)$ is a non-empty word, we denote by $\pi \alpha$ (respectively $\varpi \alpha$) the word $(\alpha_1,\ldots \alpha_{n-1})$ (resp. $(\alpha_2,\ldots \alpha_{n})$) obtained by erasing the last (resp. first) letter of $\alpha$. We use the convention $\pi \emptyset = \varpi \emptyset = \emptyset$.

Given any $X \in S$, we define the function $\epsilon_X : \mathcal{A} \to \{-1,0,1\}$ by
\[
  \epsilon_X(\alpha)=\ind{\alpha \in \mathcal{P}_X}-\ind{\varpi \alpha \in \mathcal{P}_X},
\]
where $\mathcal{P}_X$ is the set of non-empty words $\beta$ such that, starting from $X$ and applying successively the moves $\Phi_{\beta_1},\ldots,\Phi_{\beta_{L(\beta)}}$, the last move $\Phi_{\beta_{L(\beta)}}$ results in placing a ball in a previously empty bin.

For $X \in S$ and $\alpha \in \calA$, we denote by $X^{\alpha}$ the configuration of the infinite-bin model obtained after applying successively moves of type $\alpha_1, \alpha_2, \ldots \alpha_n$ to the initial configuration $X$, i.e.
\[
  X^\alpha = \Phi_{\alpha_{L(\alpha)}} \big( \Phi_{\alpha_{L(\alpha)-1}} \big( \cdots \Phi_{\alpha_2} \big( \Phi_{\alpha_1} \big( X\big) \big) \cdots \big) \big),
\]
and we set $d_X(\alpha) = B(X^\alpha,1) - B(X,1)$ the displacement of the front of the infinite-bin model after performing the sequence of moves in $\alpha$. Using this definition, we obtain an alternative expression for $\epsilon_X(\alpha)$.

\begin{lemma}
\label{lem:alternativeDefinition}
For any $\alpha \in \mathcal{A}$, we have
\begin{equation}
  \label{eqn:defEpsilon}
  \epsilon_X(\alpha) = d_X(\alpha) - d_X(\pi \alpha) - d_X(\varpi \alpha) + d_X( \pi \varpi \alpha).
\end{equation}
\end{lemma}

\begin{proof}
Observe that $d_X(\alpha) - d_X(\pi \alpha)$ equals $0$ (resp. $1$) if the last move of $\alpha$ adds a ball in a previously non-empty (resp. empty) bin. Therefore we have $d_X(\alpha) - d_X(\pi \alpha)=\ind{\alpha \in \mathcal{P}_X}$. Similarly, $d_X(\varpi \alpha) - d_X( \pi \varpi \alpha)=\ind{\varpi \alpha \in \mathcal{P}_X}$. We conclude that
\[
\epsilon_X(\alpha) = \ind{\alpha \in \mathcal{P}_X}-\ind{\varpi \alpha \in \mathcal{P}_X}=d_X(\alpha) - d_X(\pi \alpha) - d_X(\varpi \alpha) + d_X( \pi \varpi \alpha).\text{\qedhere}
\]\end{proof}

As a direct consequence of Lemma~\ref{lem:alternativeDefinition}, for any $\alpha = (\alpha_1, \ldots \alpha_n)$ we have
\begin{equation}
  \label{eqn:inverseEpsilon}
  d_X(\alpha) = \sum_{k=1}^n \sum_{j=1}^{n-k+1} \epsilon_X((\alpha_k,\alpha_{k+1},\ldots,\alpha_{k+j-1})),
\end{equation}
i.e., the displacement induced by $\alpha$ is the sum of $\epsilon(\beta)$ for any consecutive subword $\beta$ of $\alpha$ (where the subwords $\beta$ are counted with multiplicity).

\begin{remark}
One could also go the other way round, start with $d_X$ and define $\epsilon_X$ to be the function verifying
\[
\forall\alpha\in\mathcal{A},d_X(\alpha)=\sum_{\beta\prec\alpha}{\epsilon_X(\beta)m(\beta,\alpha)},
\]
where $\beta\prec\alpha$ denotes the fact that $\beta$ is a factor of $\alpha$ (i.e. a consecutive subword of $\alpha$) and $m(\beta,\alpha)$ denotes the number of times $\beta$ appears as a factor of $\alpha$. In that case, one would obtain formula~\eqref{eqn:defEpsilon} for $\epsilon_X$ as the result of a M\H{o}bius inversion formula (see~\cite[Sections 3.6 and 3.7]{St} for details on incidence algebras and M\H{o}bius inversion formulas).
\end{remark}

Using these notation and results, we prove the following lemma.
\begin{lemma}
\label{lem:seriesformula}
For any probability measure $\mu$ and $X \in S$, we have
\[
  v_\mu = \lim_{n \to +\infty} \frac{1}{n} \sum_{k=1}^n \sum_{\alpha \in \mathcal{A} : L(\alpha) \leq k} \epsilon_X(\alpha) W_\mu(\alpha).
\]
\end{lemma}

This lemma straightforwardly implies Theorem~\ref{thm:seriesformula} by Stolz-Ces\`aro theorem.

\begin{proof}
Let $(X_n)$ be an IBM($\mu$) starting from the configuration $X \in S$. We have, by definition of $d_X$, $d_X((\xi_1,\ldots \xi_n)) = B(X_n,1) - B(X_0,1)$. Moreover, by Theorem~\ref{thm:existsSpeed} and dominated convergence,
\[
  \lim_{n \to +\infty} \frac{1}{n} \E\left(  d_X((\xi_1,\ldots \xi_n)) \right) = v_\mu.
\]
We easily compute $\E\left(  d_X((\xi_1,\ldots \xi_n)) \right)$ using \eqref{eqn:inverseEpsilon}, we obtain
\begin{align*}
  \E\left(  d_X((\xi_1,\ldots \xi_n)) \right) &= \sum_{k=1}^n \sum_{j=1}^{n-k+1} \E\left(\epsilon_X((\xi_k,\xi_{k+1},\ldots,\xi_{k+j-1}))\right)\\
  &= \sum_{k=1}^n \sum_{j=1}^{n-k+1} \sum_{\alpha \in \mathcal{A} : L(\alpha) = j} W_\mu(\alpha) \epsilon_X(\alpha)\\
  &= \sum_{k=1}^n \sum_{\alpha \in \mathcal{A}:L(\alpha) \leq k}  W_\mu(\alpha) \epsilon_X(\alpha),
\end{align*}
which concludes the proof.
\end{proof}

In Section~\ref{sec:powerseries}, we study in more details the function $\epsilon_X$. In particular, we give sufficient conditions on $\alpha$ to have $\epsilon_X(\alpha)=0$, which allows to prove that in some cases, the series $\sum_{\alpha \in \mathcal{A}} \epsilon_X(\alpha) W_\mu(\alpha)$ is absolutely convergent.

\section{Length of the longest path in Barak-Erd\texorpdfstring{\H{o}}{o}s graphs}
\label{sec:geo}

In the rest of the article, we use the results obtained in the previous sections to study the asymptotic behaviour of the length of the longest path in a Barak-Erd\H{o}s graph. Let $p \in [0,1]$, we write $\mu_p$ for the geometric distribution on $\N$ with parameter $p$, verifying $\mu_p(k) = p(1-p)^{k-1}$ for any $k \geq 1$. In this section, we present a coupling introduced by Foss and Konstantopoulos~\cite{FK} between an IBM($\mu_p$) and a Barak-Erd\H{o}s graph of size $n$, used to compute the asymptotic behaviour of the length of the longest path in this graph.

Recall that a Barak-Erd\H{o}s graph on the $n$ vertices $\{1,\ldots, n\}$ with edge probability $p$ is constructed by adding an edge from $i$ to $j$ with probability $p$, independently for each pair $1\leq i<j\leq n$. We write $L_n(p)$ for the length of the longest path in this graph. Newman~\cite{N} proved that $L_n$ increases at linear speed. More precisely, there exists a function $C$ such that for any $p \in [0,1]$,
\[
  \lim_{n \to +\infty} \frac{L_n(p)}{n} = C(p) \quad \text{in probability.}
\]
Moreover, he proved that $C(p)$ is continuous and increasing on $[0,1]$, and that $C'(0)=e$.

Let $p \in (0,1)$ and $(X_n)$ be an IBM($\mu_p$), we set $v_p = v_{\mu_p}$ the speed of $(X_n)$, which is well-defined by Proposition~\ref{prop:speedIBM}. Foss and Konstantopoulos~\cite{FK} observed, through a coupling between this IBM and the Barak-Erd\H{o}s graph, that
\begin{equation}
\label{eq:coupling}
  C(p) = v_p = \lim_{n \to +\infty} \frac{B(X_n,1)}{n} \quad \text{a.s.}
\end{equation}

We now construct the coupling used to derive~\ref{eq:coupling}. We associate an infinite-bin model configuration in $S$ to each acyclic directed graph on vertices $\{1, \ldots, n\}$ as follows: for each vertex $1 \leq i \leq n$, we add a ball in the bin indexed by the length of the longest path ending at vertex $i$, and infinitely many balls in bins with negative index (see Figure~\ref{fig:BEtoIBM} for an example). We denote by $\ell_i$ the length of the longest path ending at position $i$.

\begin{figure}[htbp]
\centering
\subfloat[An acyclic directed graph $G$.]{\includegraphics[width=5.5cm]{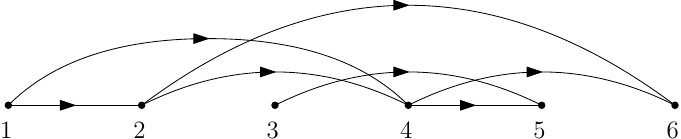}}
\hspace{1cm}\subfloat[The infinite-bin model configuration corresponding to this graph.]{\includegraphics[width=4.6cm]{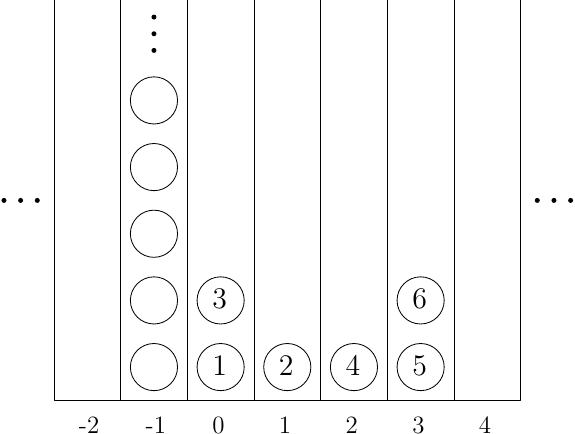}}
\caption{From a Barak-Erd\H{o}s graph to an infinite-bin model configuration.}
\label{fig:BEtoIBM}
\end{figure}

We now construct the Barak-Erd\H{o}s graph as a dynamical process, which is run in parallel with its associated infinite-bin model. At time $n=0$, we start with the Barak-Erd\H{o}s graph with no vertex, the empty graph, and the infinite-bin model with infinitely many balls in bins of negative index, and no ball in other bins (which is called configuration $Y_0$). At time $n=1$, we add vertex $1$ to the Barak-Erd\H{o}s graph. As $\ell_1=0$, we also add a ball in the bin of index $0$ to the configuration $Y_0$, to obtain the configuration $Y_1$.

At time $n>1$, we add vertex $n$ to the Barak-Erd\H{o}s graph on $\{1,\ldots, n-1\}$. We compute the law of $\ell_n$ conditionally on $(\ell_i, i \leq n-1)$. Let $\sigma$ be a permutation of $\{1,\ldots,n-1\}$ such that $\ell_{\sigma(1)} \geq \ell_{\sigma(2)} \geq \cdots \geq \ell_{\sigma(n-1)}$. The permutation is not necessarily uniquely defined by these inequalities, but this does not matter for our purpose. For each $1 \leq i \leq n-1$, there is an edge between $n$ and $\sigma(i)$ with probability $p$, independently of any other edge. In this case, there is a path of length $\ell_i + 1$ in the Barak-Erd\H{o}s graph that end at site $n$. The smallest number $\xi_n$ such that there is an edge between $\sigma(\xi_n)$ and $n$ is distributed as a geometric random variable, where if $\xi_n >n-1$, then there is no edge between $n$ and a previous vertex, thus $\ell_n = 0$ and we add a ball at position $0$. As a consequence, the state associated to the graph of size $n$ is given by $Y_n = \Phi_{\xi_n}(Y_{n-1})$.

We have coupled the IBM($\mu_p$) $(Y_n)$ with a growing sequence of Barak-Erd\H{o}s graphs, in such a way that for any $n \in \N$, the length of the longest path in the Barak-Erd\H{o}s graph of size $n$ is given by $B(Y_n,1)$. Therefore, \eqref{eq:coupling} is a direct consequence of Proposition~\ref{prop:speedIBM}.

We now use \eqref{eqn:encadrementSpeed} to bound the function $C$. We recall from the introduction that in \cite{FK}, Foss and Konstantopoulos obtained upper and lower bounds for $C(p)$, that are tight enough for $p$ close to $1$ to give the first five terms of the Taylor expansion of $C$ around $p=1$ (see \eqref{eqn:taylorFK}). We use measures with finite support to approach $\mu_p$, as in the proof of Proposition~\ref{prop:speedIBM}. We obtain two sequences of functions that converge exponentially fast toward $C$ on $[\epsilon,1]$ for any $\epsilon>0$. Let $k \geq 1$, we set
\[
  \underline{\mu}_p^k(\{j\}) = p(1-p)^{j-1}\ind{j \leq k} \,\, \text{and} \,\, \bar{\mu}_p^k(\{j\}) = p(1-p)^{j-1} \ind{j \leq k} + (1-p)^k\ind{j=k}.
\]
We write $\underline{C}_k(p) = v_{\underline{\mu}_p^k}$ and $\bar{C}_k(p) = v_{\bar{\mu}_p^k}$. By \eqref{eqn:encadrementSpeed}, for any $k \geq 1$ we have $\underline{C}_k(p) \leq C(p) \leq \bar{C}_k(p)$. Moreover, as a (very crude) upper bound, for any $p \in [0,1]$ we have
\begin{equation}
  \label{ean:cvExp}
  0 \leq \bar{C}_k(p) - \underline{C}_k(p) \leq (1-p)^k,
\end{equation}
see Remark~\ref{rem:encadrementSpeed}. Hence $\bar{C}_k - \underline{C}_k$ converges uniformly to the zero function at an exponential rate on any interval of the form $[\epsilon,1]$, with $\epsilon > 0$. Moreover, note that $\bar{C}_k(0) = 1/k$ and $\underline{C}_k(0) = 0$. Since the sequence $(\bar{C}_k - \underline{C}_k)_k$ is decreasing, by Dini's theorem it converges uniformly on $[0,1]$ to the zero function.

Using Proposition~\ref{prop:speedfinite}, the functions $\underline{C}_k$ and $\bar{C}_k$ can be explicitly computed. For example, taking $k=3$ we obtain
\[ \tfrac{p \left(p^2-3 p+3\right)^2 \left(p^4-6 p^3+14 p^2-16 p+8\right)}{3 p^6-26 p^5+96 p^4-196 p^3+235 p^2-158 p+47}
  \leq C(p)\leq \tfrac{p^3-2 p^2+p-1}{p^5-4 p^4+8 p^3-9 p^2+6 p-3}.
\]
For any $k \in \N$, $\underline{C}_k$ and $\bar{C}_k$ are rational functions of $p$. Their convergence toward $C$ is very fast, which enables to bound values of $C(p)$. For instance, taking $k=9$, we obtain $C(0.5) = 0.5780338 \pm 2.10^{-8}$, improving $C(0.5) = 0.58\pm 10^{-2}$ given by the bounds in~\cite{FK}.

\begin{figure}[htbp]
\centering
\subfloat[$\underline{C}_3$ and $\bar{C}_3$.]{\includegraphics[width=4cm]{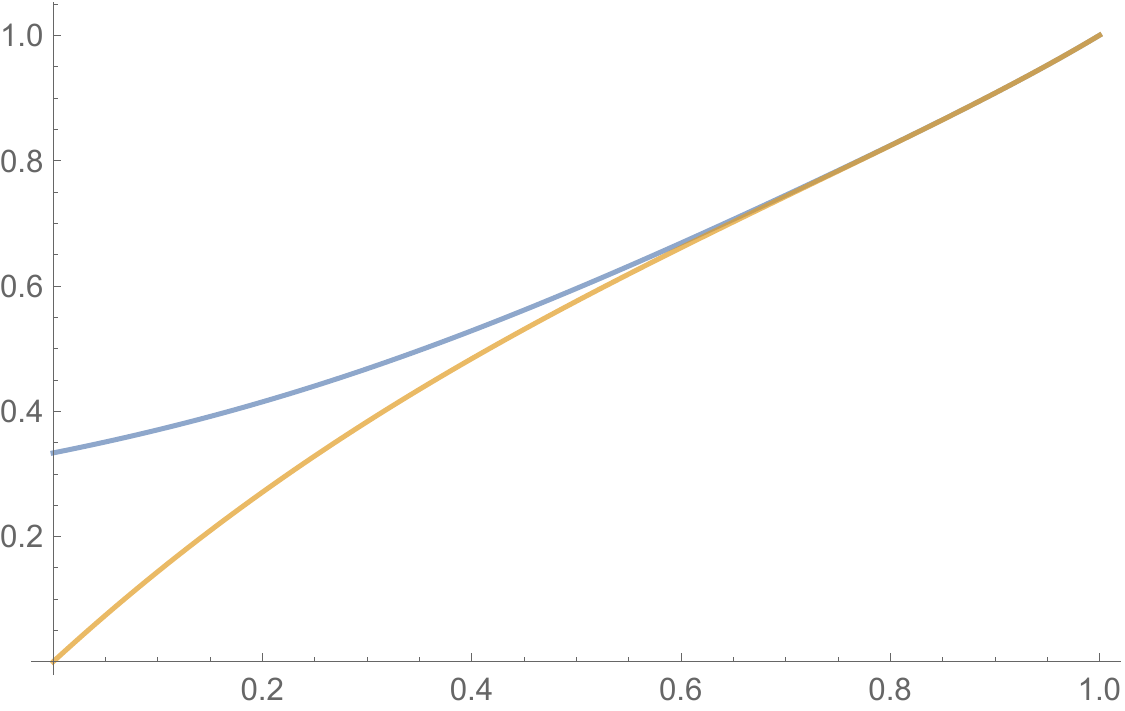}}
\hspace{\stretch{1}}\subfloat[$\underline{C}_6$ and $\bar{C}_6$.]{\includegraphics[width=4cm]{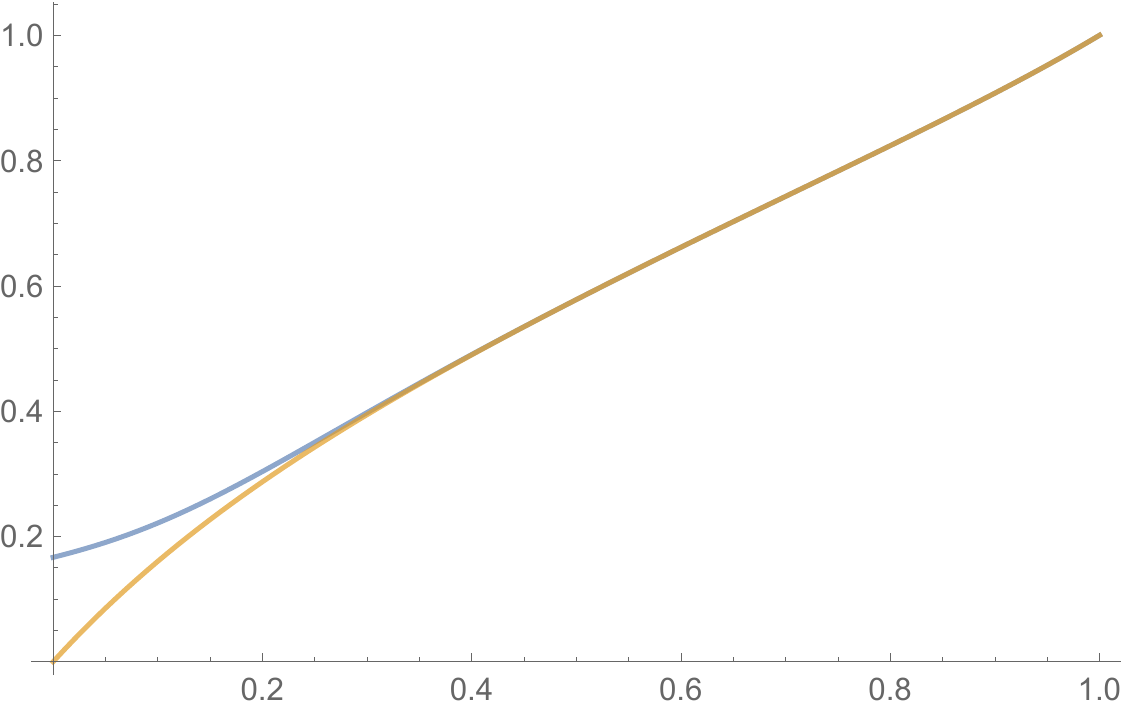}}
\hspace{\stretch{1}}\subfloat[$\underline{C}_9$ and $\bar{C}_9$.]{\includegraphics[width=4cm]{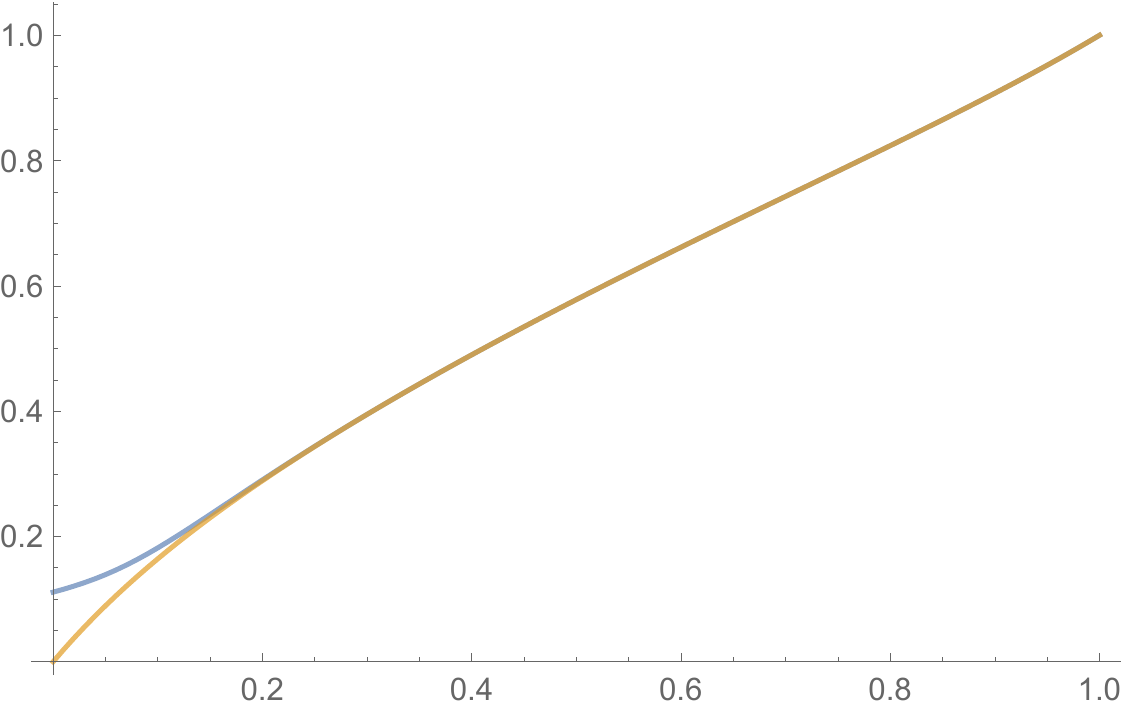}}
\caption{Lower and upper bounds $\underline{C}_k$ and $\bar{C}_k$ for $C$, for $k\in\left\{3,6,9\right\}$.}
\label{fig:geometricspeed}
\end{figure}

The functions $\underline{C}_k$ and $\bar{C}_k$ are very close for $p$ close to 1, which enables to compute the Taylor expansion of $C(1-q)$ to any order as $q \to 0$. For example, comparing the Taylor expansion of $\underline{C}_6$ and $\bar{C}_6$, we obtain the first 14 terms of the Taylor expansion of $C$. However, Theorem~\ref{thm:powerseries} gives another way to obtain this Taylor expansion.

\section{Power series expansion of \texorpdfstring{$C$}{\textit{C}} in dense graphs}
\label{sec:powerseries}

In this section, we prove that $C$ is analytic for $p>1/2$. Recall that for any word $\alpha \in \mathcal{A}$, we defined the height of $\alpha$ to be
\[
H(\alpha)=\sum_{i=1}^{L(\alpha)}{\alpha_i}-L(\alpha).
\]
For $p \in [0,1]$ we set
\[
  W_p(\alpha):= W_{\mu_p}(\alpha) = p^{L(\alpha)} (1-p)^{H(\alpha)}.
\]
By Theorem~\ref{thm:seriesformula}, if $\sum_{\alpha \in \mathcal{A}} |\epsilon_X(\alpha)| W_p(\alpha)<+\infty$, then we have
\begin{equation}
  \label{eqn:defineSeries}
  C(p) = \sum_{\alpha \in \mathcal{A}} \epsilon_X(\alpha) p^{L(\alpha)} (1-p)^{H(\alpha)}.
\end{equation}
We first prove that this series is absolutely convergent. To do so, we obtain sufficient conditions on $\alpha$ to have $\epsilon_X(\alpha)=0$. We say that a word $\alpha=(\alpha_1,\ldots,\alpha_l)$ has a renovation event at position $n\geq1$ if for all $0\leq k\leq l-n$, $\alpha_{n+k}\leq k+1$. This concept appeared first in~\cite{Bor}, then in~\cite{FK} where these events are used to create time intervals on which the process starts over and is independent of its past. We first show that the existence of a renovation event in $\alpha$ implies $\epsilon_X(\alpha)=0$.

\begin{lemma}
\label{lem:renovation}
Let $X \in S$, if $\alpha\in \calA$ with $L(\alpha) \geq 2$ has a renovation event at position $n\geq2$, then $\epsilon_X(\alpha)=0$.
\end{lemma}

\begin{proof}
Let $\alpha \in \calA$ be a word of length $l$ with a renovation event at position $n \geq 2$. When we run $\alpha$ starting from the configuration $X$, the move $\alpha_n=1$ creates a ball in a previously empty bin, of index say $b$.

As $\alpha_{n+k}\leq k+1$ for all $0\leq k\leq l-n$, we are capable of placing the balls produced by these moves in bins of index $b$ or greater, without knowing any information about the bins to the left of bin $b$ (except for the fact that the bin $b-1$ contains at least one ball).

When we run $\varpi \alpha$ starting from $X$, the move $\alpha_n$ again creates a ball in a previously empty bin, of index say $b'$. Running the moves $\alpha_{n+1},\ldots,\alpha_l$ will produce the same construction as when we run $\alpha$, with everything just shifted by $b'-b$. In particular, the last move of $\alpha$ places a ball in a previously empty bin if and only if the last move of $\varpi \alpha$ places a ball in a previously empty bin. Consequently $\ind{\alpha \in \mathcal{P}_X}=\ind{\varpi \alpha \in \mathcal{P}_X}$ so $\epsilon_X(\alpha)=0$.
\end{proof}

Using Lemma~\ref{lem:renovation}, we are able to prove that for all $k \in \N$, the set of words of height smaller than $k$ such that $\epsilon_X(\alpha) \neq 0$ is finite.

\begin{lemma}
\label{lem:lengthsmallerheight}
Let $X \in S$, for any $\alpha\in \calA$ such that $L(\alpha) > H(\alpha)+1$, we have $\epsilon_X(\alpha)=0$.
\end{lemma}

\begin{proof}
Let $\alpha$ be a word $(\alpha_1,\ldots,\alpha_l)$ such that $l=L(\alpha)>H(\alpha)+1$. For any $1\leq k\leq l$, define $S(k)=\sum_{i=1}^k(\alpha_i-2)$. As $L(\alpha)>H(\alpha)+1$ we have $S(l)<-1$. We set $n=\min\left\{k:S(t)<-1 \ \forall t\geq k\right\}$.

Observe that we have $S(1)=\alpha_1-2\geq-1$, thus $n\geq 2$. By induction, for any $0\leq k\leq l-n$, we have $S(n+k)\geq-k-2$ and $\alpha_{n+k}\leq k+1$. Thus $\alpha$ has a renovation event at position $n\geq 2$, so $\epsilon_X(\alpha)=0$ by Lemma~\ref{lem:renovation}.
\end{proof}

Using Lemma~\ref{lem:lengthsmallerheight}, we prove the absolute convergence of the series in \eqref{eqn:defineSeries}.

\begin{lemma}
\label{lem:defineTildeC}
Let $X\in S$. The series $\sum_{\alpha \in \calA} |\epsilon_X(\alpha)| W_p(\alpha)$ converges for all $p>1/2$.
\end{lemma}

\begin{proof}
Let $p > 1/2$. Define $\calA_l^h$ to be the set of words of length $l$ and height $h$. Observe that $\calA_l^h$ is the set of compositions of the integer $h+l$ into $l$ parts and it is well-known that $\#\calA_l^h=\binom{h+l-1}{l-1}$. By Lemma~\ref{lem:lengthsmallerheight}, if $\alpha$ is a word such that $|\epsilon_X(\alpha)|=1$, then $L(\alpha)\leq H(\alpha)+1$, thus
\[
\sum_{\alpha \in \calA} |\epsilon_X(\alpha)| W_p(\alpha) \leq \sum_{h \geq 0} \sum_{l=1}^{h+1} \sum_{\alpha\in\calA_l^h} W_p(\alpha).
\]

By definition of $W_p(\alpha)$, we have
\begin{align*}
\sum_{\alpha \in \calA} |\epsilon_X(\alpha)| W_p(\alpha) &\leq \sum_{h\geq0} \sum_{l=1}^{h+1} p^l(1-p)^h\#\calA_l^h  \\
&\leq p\sum_{h\geq0} \sum_{l=0}^{h} p^l(1-p)^h\binom{h+l}{l}.
\end{align*}
Let $\left(S_n\right)$ be a random walk on $\Z$ starting at $0$ and doing a step $+1$ (resp. $-1$) with probability $p$ (resp. $1-p$). Then for all $p>1/2$, we have
\begin{align*}
\sum_{h\geq0} \sum_{l=0}^{h} p^l(1-p)^h\binom{h+l}{l} &= \sum_{n\geq0} \sum_{l=0}^\floor{n/2} \binom{n}{l} p^l (1-p)^{n-l} \\
&=\sum_{n\geq0} \sum_{l=0}^\floor{n/2} \P\left(S_n = 2l - n\right) 
= \sum_{n\geq0} \P\left(S_n \leq 0\right) < +\infty.
\end{align*}
Indeed, we have $\E(S_1) = 2p-1>0$, and $\P(S_n \leq 0)$ decays exponentially fast by Cram\'er's large deviations theorem.
\end{proof}

Using the above lemma and Theorem~\ref{thm:seriesformula}, we immediately obtain the following result.
\begin{lemma}
\label{lem:estimateSpeed}
For any $X \in S$ and $p>1/2$, \eqref{eqn:defineSeries} holds.
\end{lemma}

We use this formula for $C$ to prove that the function can be written as a power series around every $p > 1/2$.
\begin{proof}[Proof of Theorem~\ref{thm:powerseries}]
Fix $\tfrac{1}{2}<p \leq r \leq 1$ and write $x=r-p\geq 0$. We write $C(p)=C(r-x)$ as a power series in $x$ and determine its radius of convergence.
\begin{align*}
C(p)&= \sum_{\alpha \in \calA} \epsilon_X(\alpha)p^{L(\alpha)}(1-p)^{H(\alpha)} \\
&=\sum_{\alpha \in \calA} \epsilon_X(\alpha)(r-x)^{L(\alpha)}(1-r+x)^{H(\alpha)} \\
&=\sum_{\alpha \in \calA} \epsilon_X(\alpha)\sum_{i=0}^{L(\alpha)}\binom{L(\alpha)}{i}(-1)^{i}x^ir^{L(\alpha)-i}\sum_{j=0}^{H(\alpha)}\binom{H(\alpha)}{j}x^j(1-r)^{H(\alpha)-j} .
\end{align*}
Taking absolute values inside the last series, we obtain
\begin{align*}
&\sum_{\alpha \in \calA} |\epsilon_X(\alpha)|\sum_{i=0}^{L(\alpha)}\binom{L(\alpha)}{i}x^ir^{L(\alpha)-i}\sum_{j=0}^{H(\alpha)}\binom{H(\alpha)}{j}x^j(1-r)^{H(\alpha)-j} \\
&=\sum_{\alpha \in \calA} |\epsilon_X(\alpha)|(r+x)^{L(\alpha)}(1-r+x)^{H(\alpha)} \\
&=\sum_{\alpha \in \calA} |\epsilon_X(\alpha)|(2r-p)^{L(\alpha)}(1-p)^{H(\alpha)}.
\end{align*}
By the same computations as in Lemma~\ref{lem:defineTildeC}, we have
\[
\sum_{\alpha \in \calA} |\epsilon_X(\alpha)|(2r-p)^{L(\alpha)}(1-p)^{H(\alpha)} \leq (1-p)\sum_{h \geq 0}\sum_{l=0}^h (2r-p)^l(1-p)^h\binom{h+l}{l}.
\]
If this quantity is finite, then the power series expansion of $C$ around $r$ has a radius of convergence at least $r-p$. Writing $(S_n^{p,r})$ for a random walk on $\Z$ starting at $0$ and doing a step  $+1$ (resp. $-1$) with probability $\tfrac{2r-p}{2r+1-2p}$ (resp. $\tfrac{1-p}{2r+1-2p}$), we have
\begin{equation}
\label{eq:geometricdomination}
\sum_{h \geq 0}\sum_{l=0}^h (2r-p)^l(1-p)^h\binom{h+l}{l}=\sum_{n\geq0}(2r+1-2p)^n \P\left(S_n^{p,r}\leq0\right).
\end{equation}
By Chernoff's bound, we obtain
\begin{align*}
\P\left(S_n^{p,r}\leq0\right)&\leq \inf_{t>0} \left(\E\left[e^{-tS^{p,r}_1}\right]\right)^n \\
&\leq \inf_{t>0} \left(\frac{2r-p}{2r+1-2p}e^{-t}+\frac{1-p}{2r+1-2p}e^t\right)^n \\
&\leq \left(\frac{2\sqrt{(2r-p)(1-p)}}{2r+1-2p}\right)^n.
\end{align*}
Thus the series in~\eqref{eq:geometricdomination} converges as soon as $2\sqrt{(2r-p)(1-p)}<1$, i.e. if
\[r+\tfrac{1}{2}-\sqrt{r^2-r+\tfrac{1}{2}} <p \leq r \leq1.\]

For $r>1/2$, we have $r+\tfrac{1}{2}-\sqrt{r^2-r+\tfrac{1}{2}}<r$, thus the power series expansion of $C$ centered at $r$ has a positive radius of convergence. Therefore $C$ is analytic on $\left(\tfrac{1}{2},1\right]$. In particular, for $r=1$, expanding the expression in \eqref{eqn:defineSeries} in powers of $(1-p)$, we conclude that for $p$ larger than $\frac{3-\sqrt{2}}{2}$, we have
\[
  C(p) = \sum_{k \geq 0} a_k (1-p)^k,
\]
with $(a_k)$ defined in \eqref{eqn:formulaCoeff}.
\end{proof}

\section{Longest directed path in sparse graphs}
\label{sec:smallp}

We study in this section the asymptotic behaviour of $C(p)$ as $p \to 0$. Newman proved in \cite{N} that $C(p) \sim pe$. We link in Section~\ref{subsec:proofThm} this result with the estimate obtained by Aldous and Pitman \cite{AP} for the speed of an IBM with uniform distribution. Let $k \in \N$, we write $\nu_k$ for the uniform distribution on $\{1, \ldots, k\}$ and $w_k$ for the speed of the IBM($\nu_k$), Aldous and Pitman proved that
\begin{equation}
  \label{eqn:AP}
  (kw_k, k \in \N) \text{ increases toward } e \text{ as } k \to +\infty.
\end{equation}
This result is obtained by observing that the IBM($\nu_k$) can be coupled with a continuous-time branching random walk with selection.

Recent developments were obtained on the asymptotic behaviour of the speed of a discrete-time branching random walk with selection. This behaviour was conjectured by Brunet and Derrida \cite{BD97}, and proved recently by B\'erard and Gou\'er\'e \cite{BeG10}. The result of Bérard and Gouéré was extended by Mallein \cite{Mal15a,Mal15b} to more general discrete-time branching random walks. In discrete-time branching random walks with selection, multiple reproduction events may occur at the same time, while in the infinite-bin model, which is also a discrete-time process, only one reproduction event occurs at each time step.

We thus consider the infinite-bin model as the pure jump process of a continuous-time particle system in which a move of type $k$ happens at rate $\mu(k)$. This particle system can be coupled with a continuous-time branching random walk with selection. In particular, the IBM($\nu_k$) corresponds to the jump process of a system of $k$ particles in which every particle gives birth to a child at rate $1/k$, which is put one step to its right. Simultaneously, the leftmost particle is removed from the process. We extend the results obtained for discrete-time branching random walks to continuous-time versions, proving in this section the following estimate.
\begin{lemma}
\label{lem:AP}
We have $kw_k = e - \frac{\pi^2 e}{2} (\log k)^{-2} (1 + o(1))$ as $k \to +\infty$.
\end{lemma}

Applying Lemma~\ref{lem:AP} to bound compute the asymptotic behaviour of $C$ as $p \to 0$, we are able to prove Theorem~\ref{thm:smallp} :
\begin{equation}
\label{eqn:asymptoticBehaviour}
C(p) = e p \left( 1 - \frac{\pi^2 }{2} (-\log p)^{-2}\right)  + o(p (-\log p)^{-2}) \text{ as } p \to 0.
\end{equation}

The rest of the section is organized as follows. In Section~\ref{subsec:proofThm}, we prove Theorem~\ref{thm:smallp} assuming Lemma~\ref{lem:AP}. In Section~\ref{subsec:proofLem}, we prove Lemma~\ref{lem:AP} assuming that the Brunet--Derrida behaviour of continuous-time branching random walks with selection is known. Preliminary results on continuous-time branching random walks with selection are derived in Section~\ref{subsec:estBrw} and  the speed of the cloud of particles in a continuous-time branching random walk with selection is finally obtained in Section~\ref{subsec:proofBrw}, completing the proof of Theorem~\ref{thm:smallp}.

\subsection{Proof of Theorem~\ref{thm:smallp} assuming Lemma~\ref{lem:AP}}
\label{subsec:proofThm}

We use the increasing coupling of Proposition~\ref{prop:speedIBM} to link the asymptotic behaviours of $w_k$ and $C(1/k)$ as $k \to +\infty$.

\begin{lemma}
\label{lem:coupp}
For any $k \in \N$ we have
\begin{align*}
  &\forall p \in [\tfrac{1}{k+1},\tfrac{1}{k}], \, C(p) \leq w_k\\
  &\forall p \in [0,1],\, C(p) \geq kp(1-p)^k w_k.
\end{align*}
\end{lemma}

\begin{proof}
Let $k \in \N$ and $p \in [\tfrac{1}{k+1},\tfrac{1}{k}]$. We observe that for any $j \in \N$,
\[
  \mu_p([1,j]) = \sum_{i=1}^j p(1-p)^{i-1} \leq (pj) \wedge 1 \leq \nu_k([1,j]).
\]
Therefore $C(p) \leq w_k$ by \eqref{eqn:encadrementSpeed}.

Let $p \in [0,1]$, we set $x = k p (1-p)^{k-1}$. Observe that $0\leq x\leq 1$. For any $j \in \N$, we have
\[
  \mu_p([1,j]) = \sum_{i=1}^j p(1-p)^{i-1} \geq (j\wedge k) p(1-p)^{k-1} \geq k \nu_k([1,j]) p(1-p)^{k-1}.
\]
Therefore, writing $\nu_k^x = x \nu_k + (1 - x) \delta_\infty$, we have $\mu_p([1,j])\geq \nu_k^x([1,j])$ for any $j \in \N$. We apply \eqref{eqn:encadrementSpeed} to $\mu_p$ and $\nu_k^x$. By Lemma~\ref{lem:obvious}, the speed of the IBM($\nu_k^x$) is $xw_k$. We conclude that for any $k \in \N$ and $p \in [0,1]$, we have
\[C(p) \geq k p (1-p)^{k-1} w_k. \qedhere \]
\end{proof}

We now prove Theorem~\ref{thm:smallp} assuming that Lemma~\ref{lem:AP} holds.

\begin{proof}[Proof of Theorem~\ref{thm:smallp}]
For any $k \in \N$ and $p \in [\frac{1}{k+1},\frac{1}{k}]$, by Lemma~\ref{lem:coupp}, we have $C(p)/p \leq (k+1)w_k$, therefore Lemma~\ref{lem:AP} yields
\[
  \limsup_{p \to 0} (\log p)^2\left(\frac{C(p)}{p} - e\right) \leq \limsup_{k \to +\infty} (\log k)^2 \left((k+1)w_k - e \right) \leq - \frac{\pi^2 e}{2}.
\]

By Lemma~\ref{lem:coupp} again, we have $C(p)/p \geq (1-p)^k (k w_k)$ for any $k \in \N$ and $p \in [0,1]$. Let $\delta >0$, we set $k = \ceil{1/p^{1-\delta}}$. Then
\[  (\log p)^2 \left( \frac{C(p)}{p} - e\right) \geq \frac{(\log k)^2}{(1-\delta)^2} \left( (1-p)^k (k w_k) - e \right). \]
Using again Lemma~\ref{lem:AP} and the fact that $(1-p)^k - 1 \sim -p^\delta$ as $p \to 0$, we have
\[
  \liminf_{p \to 0} (\log p)^2\left(\frac{C(p)}{p} - e\right) \geq - \frac{\pi^2 e}{2(1-\delta)^2}.
\]
Letting $\delta \to 0$ concludes the proof.
\end{proof}

\subsection{Proof of Lemma~\ref{lem:AP} using branching random walks}
\label{subsec:proofLem}

As said in the introduction to the section, to obtain the asymptotic behaviour of its speed, Aldous and Pitman compared the IBM($\nu_k$) with a continuous-time branching random walk with selection, that we now define more precisely. Let $k \in \N$, we define a continuous-time system of $k$ particles on $\Z$ as follows. At time $0$, the positions of particles are ranked in a non-increasing order as  $Y^k_0(1) \geq Y^k_0(2) \geq \cdots \geq Y^k_0(k)$. Particles stay in place for all their lifetime. Each particle, independently of all others, reproduces at rate $1$. A the first reproduction time $t$, the parent particle creates a new daughter particle one step to its right. Simultaneously, the leftmost particle is erased so that the total number of particles remains equal to $k$. The positions of particles are then updated as $Y^k_t(1) \geq \cdots \geq Y^k_t(k)$, setting $Y^k_s(j) = Y^k_0(j)$ for $j \leq k$ and $s < t$. After this reproduction event, particles in the process continue to reproduce and be deleted according to the same procedure.

The process $Y^k$ is called a (continuous-time) branching random walk with selection. Indeed, the particles reproduce independently of one another, but the total size of the population is capped to a fixed number $k$ by removing particles from the left at each time a new particle is born. Using proof techniques coming from discrete-time branching random walks with selection, we will show in the forthcoming sections the following estimate for the speed of the cloud of particles $(Y^k_t(j), j \leq k)$ as $t \to \infty$.
\begin{lemma}
\label{lem:vitesseBRWSpecialCase}
For all $k \in \N$, there exists $c_k \in \R$ such that
\[
  \lim_{t \to \infty} \frac{Y^k_t(1)}{t} = \lim_{t \to \infty} \frac{Y^k_t(k)}{t} = c_k \quad \text{a.s.}
\]
Moreover, we have $c_k - e \sim -\frac{\pi^2 e}{2} (\log k)^{-2}$ as $k \to \infty$.
\end{lemma}

The existence of the speed $c_k$ of the branching random walk with selection $Y^k$ is proved in Section~\ref{subsec:estBrw}, and its asymptotic behaviour as $k \to \infty$ is obtained in Section~\ref{subsec:proofBrw} by adapting the proofs used in \cite{BeG10,Mal15b}. Assuming for now that Lemma~\ref{lem:vitesseBRWSpecialCase} holds, we prove Lemma~\ref{lem:AP} using the Aldous-Pitman coupling described below.

\begin{proof}[Proof of Lemma~\ref{lem:AP}]
Let $k \in \N$. We write $(N_t,t \geq 0)$ for a Poisson process of parameter $k$ and $(X_n, n \geq 0)$ for an independent IBM($\nu_k$). For any $t>0$, we denote by $(Y_t(j), j \leq k)$ the positions of the rightmost $k$ balls in the configuration $X_{N_t}$, ranked in a non-increasing order.

We observe that $(Y_t(u), u \leq k)$ evolves as follows: every ball stays put until an exponential random time with parameter $k$. At that time $T$, a ball with index $u \leq k$ is chosen uniformly at random, a new ball is added at position $Y_T(u)+1$ and the leftmost ball is erased.

By classical properties of exponential random variables, this evolution admits the following alternative description. To each ball is associated a clock with parameter $1$. When a clock rings, the corresponding ball makes a ``child'' to the right of its current position, and the leftmost ball is erased. Therefore, the law of $(Y^k_t(u), u \leq k)$ is the same as the continuous-time branching random walk with selection described above. As a result, we deduce from Lemma~\ref{lem:vitesseBRWSpecialCase} that:
\[
  \lim_{t \to \infty} \frac{B(X_{N_t},1)}{t} = c_k \quad \text{a.s.}
\]
Using the fact that $N_t \sim k t$ a.s. as $t \to \infty$, by the law of large numbers we deduce that $k w_k = c_k$. Therefore, Lemma~\ref{lem:AP} is a direct consequence of Lemma~\ref{lem:vitesseBRWSpecialCase}.
\end{proof}

\subsection{Speed of the \texorpdfstring{$k$}{\textit{k}}-branching random walk}
\label{subsec:estBrw}

In this section, we present an increasing coupling for branching random walks with selection, introduced by Bérard and Gouéré \cite{BeG10}. This increasing coupling is similar in nature to Proposition~\ref{prop:increasingCoupling} but cannot be obtained as a straightforward corollary of it. Loosely speaking, we aim to couple here branching random walks with selection with different numbers of particles. The coupling expresses that the larger the population is in that branching process, the faster it moves to the right. To state this coupling, we extend the definition of branching random walk with selection to authorize the maximal size of the population to vary.

To do so, we first define the branching random walk without selection. This is a particle system on $\Z$ in which the particles behave independently of each other. After an exponential time of parameter $1$, a particle creates a child one step to its right. For all $t \geq 0$, we denote by $\mathcal{N}_t$ the set of particles alive at time $t$, and by $\mathcal{Y}_t(u)$ the position in $\Z$ of the particle $u \in \mathcal{N}_t$. The process $(\mathcal{Y}_t(u), u \in \mathcal{N}_t)_{t \geq 0}$ is referred to as the continuous-time branching random walk.

Let $H$ be a càdlàg integer-valued process adapted to the filtration of the branching random walk $(\mathcal{Y}_t(u), u \in \mathcal{N}_t)_{t \geq 0}$. We define the $H$-branching random walk as the following process. At time $0$, if there are more than $H_0$ particles in $\mathcal{N}_0$, we kill particles, together with their offspring, except the $H_0$ rightmost ones (with ties broken uniformly at random). Next, at each time $t$ such that the remaining number of particles in the process becomes larger than $H_t$ (either because $H_t < H_{t-}$ or because a birth occurred in the system), we kill particles (and their offspring) from the left until only $H_t$ remain. At every time $t \geq 0$, we set $Y^H_t(1), \ldots Y^H_t(H_t)$ to be the positions of the particles alive at time $t$ in this process, ranked in a non-increasing order. We set $Y^H_t(j) = -\infty$ by convention if there are less than $j$ particles alive at that time in the process. The process $(Y^H_t(j), j \leq H_t)_{t \geq 0}$ is referred to as the $H$-branching random walk, or $H$-BRW for short.

Note that if $H$ is a constant process, equal to $k \in \N$, then the process $Y^H$ is the same as $Y^k$ the branching random walk with selection defined in the previous section. The notation we chose is thus consistent. We now state the coupling process result, which is the main result of the section.

\begin{lemma}
\label{lem:coupbrw}
Let $Y^H$ be an $H$-BRW and $\tilde{Y}^K$ be a $K$-BRW. We assume that
\[
  \forall x \in \R,  \# \{ j \leq H_0 : Y^H_0(j) \geq x \} \leq \# \{ j \leq K_0 : \tilde{Y}^K_0(j) \geq x \}.
\]
Then there exists a coupling between $Y^H$ and $\tilde{Y}^K$ such that a.s. for any $t>0$, on the event $\{H_s \leq K_s, s \leq t\}$,
\begin{equation}
  \label{eqn:coupling}
  \forall x \in \R, \, \# \{ j \leq H_t : Y^H_t(j) \geq x \} \leq \# \{ j \leq K_t : \tilde{Y}^K_t(j) \geq x \}.
\end{equation}
\end{lemma}

This lemma, obtained as a straightforward adaptation of \cite[Lemma~1]{BeG10}, expresses that the partial order defined in Section~\ref{sec:generalities} is preserved by the dynamics of the branching random walk with selection, provided that the total number of particles alive remains always smaller for the smaller configuration.

\begin{proof}
The coupling procedure is the following: for any $j \leq H_t$, the $j$th rightmost particle in the processes $Y^H$ and $\tilde{Y}^K$ carry the same exponential clock governing their reproduction. We show that the first change in the composition of the population after time $0$ preserves the property \eqref{eqn:coupling}.
We write
\begin{align*}
  m &= \sup\{j \leq H_0 : Y^H_0(j) > -\infty\} \\ \text{and} \quad n &= \sup\{j \leq K_0 : \tilde{Y}^K_0(j) > -\infty\},
\end{align*}
the number of particles alive at time $0$ in $Y^H$ and $\tilde{Y}^K$ respectively. By assumption, we have $m \leq H_0$, and $m \leq n \leq K_0$ and for all $j \leq m$, $Y^H_0(j) \leq \tilde{Y}^K_0(j)$.

We associate exponential clocks to particles in the processes in such a way that the particles in position $Y^H_0(j)$ and $\tilde{Y}^K_0(j)$ reproduce at the same time, for any $j \leq m$. We denote by $T_b$ (resp. $T_a$) the first time one of these particles reproduces (resp. the first time a particle located at position $\tilde{Y}_0(m+1),\ldots \tilde{Y}_0(n)$ reproduces). We also set
\[
  S = \inf\{ t >0 : H_t \neq H_0 \quad \text{or} \quad K_t \neq K_0 \} \quad \text{and} \quad R = T_a \wedge T_b \wedge S.
\]
We observe that $Y^{H}$ and $\tilde{Y}^K$ are constant processes until time $R$, that $R>0$ a.s. and that $T_a \neq T_b$ a.s.

One of three things can happen at time $R$. Firstly, if $R=T_a$, there is a reproduction event in $\tilde{Y}^K$ but not in $Y^H$. If we rank in a non-increasing order these new particles, they again satisfy the partial ordering. Moreover, as $H_R \leq K_R$, applying the selection procedure to both models preserves this partial ordering, therefore
\[
  \forall x \in \R, \, \# \{ j \leq H_R : Y^H_R(j) \geq x \} \leq \# \{ j \leq K_R : \tilde{Y}^K_R(j) \geq x \}.
\]

If $R = T_b$, then there is a reproduction event in $Y^H$ and $\tilde{Y}^K$. We use the same point process to construct the child of the particle that reproduces in each process. Once again, ranking in a non-increasing order these new particles, then applying the selection, we have
\[
  \forall x \in \R, \,  \# \{ j \leq H_R : Y^H_R(j) \geq x \} \leq \# \{ j \leq K_R : \tilde{Y}^K_R(j) \geq x \}.
\]

Finally, if $R = S \not \in \{ T_a, T_b\}$, the maximal size of at least one of the populations is modified. Even if this implies the death of some particles in $Y^H$ and/or $\tilde{Y}^K$, the property \eqref{eqn:coupling} is preserved at time $R$.

Now fix $t>0$ and assume that $H_s \leq K_s$ for every $0 \leq s \leq t$. As $H$ and $K$ are integer-valued c\`adl\`ag processes, they attain their maxima on compact sets. Therefore, they are both a.s. finite on the interval $[0,t]$, so the number of particles is a.s. finite in both processes $Y^H$ and $\tilde{Y}^K$. Thus there is a.s. a finite sequence of times $(R_k)$ smaller than $t$ such that $Y^H$ or $\tilde{Y}^K$ is modified at each time $R_k$. Using this coupling on each time interval of the form $[R_k,R_{k+1}]$ yields~\eqref{eqn:coupling}.
\end{proof}

Using this lemma, we can prove that the cloud of particles in a $k$-BRW drifts at linear speed $c_k$. Note that by the coupling described in the proof of Lemma~\ref{lem:AP}, this result can be obtained as a consequence of Theorem~\ref{thm:existsSpeed}. However, we believe the following proof to be of independent interest, as it can be generalized to more diverse continuous-time branching random walks with selection.
\begin{lemma}
\label{lem:existSpeed}
For any $k \in \N$, there exists $c_k \in \R$ such that
\[
  \lim_{t \to +\infty} \frac{Y_t^k(1)}{t} = \lim_{t \to +\infty} \frac{Y^k_t(k)}{t} = c_k \quad \text{a.s.}
\]
Moreover, if $Y^k_0(1)=Y^k_0(2)= \ldots = Y^k_0(k) = 0$, we have
\begin{equation}
\label{eqn:speedBRW}
  c_k = \inf_{t > 0} \frac{\E\left[Y_t(1)\right]}{t} = \sup_{t > 0} \frac{\E\left[Y_t(k)\right]}{t}.
\end{equation}
\end{lemma}

The proof of this lemma is adapted from \cite[Proposition 2]{BeG10}.

\begin{proof}
We prove that $(Y^k_t(1))$ is a sub-additive process. We then use Kingman's sub-additive ergodic theorem (see \cite[Theorem 4]{King2} and \cite[Theorem 9.14]{Kal02}), stating that if $(X_{s,t}, 0 \leq s \leq t)$ is a càdlàg family of random variables satisfying
\begin{align}
  &\label{eq:subadd}\forall 0 \leq s \leq t \leq u, \, X_{s,u} \leq X_{s,t} + X_{t,u} \quad \text{a.s}\\
  &\label{eq:eqdist}\forall h \geq 0, \, (X_{s+h,t+h}, 0 \leq s \leq t) \egaldistr (X_{s,t},0 \leq s \leq t)\\
  &\label{eq:ergod}\forall h \geq 0, \, (X_{s+h,t+h}, 0 \leq s \leq t) \text{ is independent of } (X_{s,t}, 0 \leq s \leq t \leq h)\\
  &\label{eq:bound}\exists A > 0, \forall t \geq 0, \, -At \leq \E(X_{0,t}) < \infty \\
  &\label{eq:integ}\E\left( \left| \sup_{0 \leq s \leq t \leq 1} X_{s,t} \right| \right) < \infty,
\end{align}
then $\gamma := \lim_{t \to \infty} \frac{1}{t} \E(X_{0,t})$ exists, is finite and is equal to $\inf_{t \geq 0} \frac{\E(X_{0,t})}{t}$ (by sub-additivity), and
\begin{equation}
  \label{eq:result}
  \lim_{t \to \infty} \frac{1}{t} X_{0,t} = \gamma \quad \text{a.s. and in } L^1.
\end{equation}
We construct on the same probability space a family $(Y^k_{s,t}(j), 0 \leq s \leq t, j \leq k)$, such that for all $s \geq 0$, $(Y^k_{s, s+t}, t \geq 0)$ is a $k$-branching random walk, and $(Y^k_{s,t}(1), 0 \leq s \leq t)$ is sub-additive.

Let $N^1,\ldots,N^k$ be $k$ i.i.d. Poisson processes with unit intensity. For all $s \geq 0$, we set $Y^k_{s,s}(1) = Y^k_{s,s}(2) = \cdots = Y^k_{s,s}(k) = 0$, i.e. particles start at position $0$ at time $s$. Then the process evolves as follows: at each time $t$ such that $N^j_t \neq N^j_{t-}$, the $j$th largest particle alive at time $t-$ in $Y^k_{s,\bullet}$ creates a new child, and the leftmost particle is erased. 

By definition, we observe that for all $s \geq 0$, $(Y^k_{s,s+t}, t \geq 0)$ is a $k$-branching random walk starting with $k$ particles at position $0$ at time $0$. In particular, \eqref{eq:eqdist} is satisfied. Moreover, $Y^k_{s,\bullet}$ is measurable with respect to the Poisson processes $(N^j_{t+s}-N^j_s, t \geq s, j \leq k)$, therefore is independent of $(Y^k_{u,v}, 0 \leq u \leq v \leq s)$. This shows \eqref{eq:ergod}, i.e. that this process is ergodic.

Moreover, one can observe that the construction described here is the same as the one given in the proof of Lemma~\ref{lem:coupbrw}. Therefore, for all $s \leq t$ this process couples the $k$-branching random walks $(Y^k_{s,t+h}, h \geq 0)$ and $(Y^k_{t,t+h}, h \geq 0)$ in such a way that for all $j \leq k$ and $h \geq 0$, one has
\[
  Y^k_{s,t+h}(j) \leq Y^k_{t,t+h}(j) + Y^k_{s,t}(1) \quad \text{a.s.}
\]
Indeed, the $k$-branching random walk $(Y^k_{s,t+h}, h \geq 0)$ is coupled with the $k$-branching random walk $(Y^k_{t,t+h}+Y_{s,t}(1), h \geq 0)$ which starts with $k$ particles at position $Y_{s,t}(1)$. In particular, we have $Y^k_{s,u}(1) \leq Y^k_{s,t}(1)+ Y^k_{t,u}(1)$ a.s., proving \eqref{eq:subadd}.

To prove the last two conditions, we observe that $Y_{s,\bullet}(1)$ increases by at most $1$ at each time one of the Poisson processes jumps. Moreover, $t \mapsto Y_{s,t}$ is non-decreasing, thus, for all $t \geq 0$,
\[
  \E\left( \left| \sup_{0 \leq s \leq t \leq 1} Y_{s,t}(1) \right| \right) \leq k \quad \text{and} \quad 0 \leq \E(Y_{0,t}(1)) < \infty,
\]
proving both \eqref{eq:bound} and \eqref{eq:integ}.

As a result, by Kingman's sub-additive ergodic theorem, setting
\[
   c_k =  \lim_{t \to \infty} \frac{\E\left[Y_t(1)\right]}{t} = \inf_{t > 0} \frac{\E\left[Y_t(1)\right]}{t},
\]
we have $\lim_{t \to \infty} \frac{Y_{0,t}(1)}{t} = c_k$ a.s.

With the same construction, one can observe that $(Y^k_{s,t}(k), 0 \leq s \leq t)$ is a super-additive sequence, satisfying similar integrability assumptions as $Y_{s,t}(1)$. Therefore, setting
\[
  d_k =  \lim_{t \to \infty} \frac{\E\left[Y_{0,t}(k)\right]}{t} = \sup_{t > 0} \frac{\E\left[Y_ {0,t}(1)\right]}{t},
\]
we have $\lim_{t \to \infty} \frac{Y_{0,t}(k)}{t} = d_k$ a.s. As $Y_{s,t}(k) \leq Y_{s,t}(1)$, we have $d_k \leq c_k$. We now prove these two quantities to be equal.

We define a sequence of hitting times $(T_n, n \geq 0)$ by setting $T_0=0$, and $T_{n+1}$ is the first time after time $T_n$ where the last $k$ children are born from the same particle, and that this particle was the rightmost particle before the series of branching events. The probability that the next $k$ branching events are as such is $1/k^k > 0$, therefore $T_n < \infty$ a.s. Moreover, by definition, we have $Y_{0,T_n}(1) = Y_{0,T_n}(k)$ a.s., all particles being at the same position at that time. As a result, we have
\[
  \liminf_{t \to \infty} \frac{Y_{0,t}(1) - Y_{0,t}(k)}{t} = 0 \quad \text{a.s.}
\]
proving that $c_k = d_k$, and that $\frac{Y_{0,t}(1)}{t}$ and $\frac{Y_{0,t}(k)}{t}$ have the same limit.

Finally, we consider a $k$-branching random walk $Y^k$ starting from an arbitrary initial configuration. After a finite amount of time $t$, the process contains $k$ particles. From that point on, the process can be bounded from above and from below by $k$-branching random walks starting with $k$ particles at position $Y^k_t(1)$ and $Y^k_t(k)$ respectively. Therefore, by the previous results, we also obtain
\[
  \lim_{s \to \infty} \frac{Y^k_s(1)}{s} = \lim_{s \to \infty} \frac{Y^k_s(k)}{s} = c_k \quad \text{a.s.}
\]
completing the proof.
\end{proof}

\subsection{End of the proof of Lemma~\ref{lem:vitesseBRWSpecialCase}}
\label{subsec:proofBrw}

In this section, we use Lemma~\ref{lem:coupbrw} to compare the asymptotic behaviour of the continuous-time branching random walk with selection $Y^k$ with a discrete-time branching random walk with selection. This latter model being well-studied, we are able to deduce Lemma~\ref{lem:vitesseBRWSpecialCase} from it. The discrete-time branching random walk with selection of the rightmost $k$ individuals was introduced by Brunet and Derrida in \cite{BD97} to study noisy FKPP equations. In that article, they conjecture that the cloud of particles drifts at speed $v_k$, that satisfies
\begin{equation}
  \label{eqn:bd}
  v_k - v = - \frac{\chi}{(\log k + 3 \log \log k +o(\log \log k))^2}, \text{ as } k \to \infty,
\end{equation}
for some explicit constants $v \in \R$ and $\chi>0$.

We now describe more precisely the discrete-time $k$-branching random walk. Let $k \in \N$ and $\mathcal{M}$ be the law of a point process on $\Z$. The system starts with $k$ particles on $\Z$. At each integer time $n$, every particle dies while giving birth to offspring. The children of a given individual are positioned around their parent according to an i.i.d. point process with law $\calM$. Among all the children of the individuals at generation $n$, the rightmost $k$ ones survive to form the new generation, with ties broken in an uniform fashion.

For every $n \geq 0$, we set $Z^k_n(1) \geq Z^k_n(2) \geq \cdots \geq Z^k_n(k)$ to be the ranked positions of particles alive at generation $n$ in this branching random walk with selection. To avoid the possibility of the process dying out, we assume that every individual always has at least one child, and that the mean number of children is larger than $1$. Note that the formulation of the discrete-time process is slightly different from the one of the continuous-time process, since the parents get immediately killed in the discrete-time setting but not in the continuous-time setting. We could easily adapt the definition of the continuous-time process by saying that when a particle reproduces, it has two children, one at its current location and one immediately to its right, and that the parent gets killed just after reproducing.

We now introduce some notation. Let $M$ be a point process of law $\calM$. We assume that 
\begin{equation}
  \label{eqn:finiteLT}
  \kappa(\theta) := \log \E\left( \sum_{m \in M} e^{\theta m} \right) < \infty \quad \text{ for all $\theta > 0$.}
\end{equation}
Note that the function $\kappa$ is then infinitely differentiable and strictly convex on $(0,\infty)$ as 
\[
\kappa''(\theta) = \E\left( (m - \kappa'(\theta))^2 e^{\theta m - \kappa(\theta)}   \right)>0.
\]
We assume that there exists $\theta^*>0$ such that
\begin{equation}
  \label{eqn:regula}
  \theta^* \kappa'(\theta^*) - \kappa(\theta^*) = 0,
\end{equation}
and we write $v:= \inf_{\theta>0} \frac{\kappa(\theta)}{\theta} = \frac{\kappa(\theta^*)}{\theta^*} = \kappa'(\theta^*)$ and $\sigma^2:= \kappa''(\theta^*)$.

Bérard and Gouéré studied the asymptotic behaviour of the speed of the $k$-branching random walk as $k \to \infty$ under the assumption that the point process $M$ is binary. This result was then extended by Mallein \cite{Mal15b} to more general reproduction laws. We use the following result, which is a special case of \cite[Theorem 1.1]{Mal15b}, applied to the process $(\theta^* Z^k_n - n \kappa(\theta^*), n \geq 0)$.

\begin{theorem}
\label{thm:Mallein}
Let $Z^k$ be a discrete-time branching random walk with selection of the rightmost $k$ particles, whose reproduction law satisfies \eqref{eqn:finiteLT} and \eqref{eqn:regula}. We additionally assume that
\begin{gather}
  \label{eqn:boundary} \E\left( \left|\max_{m \in M} m \right|^2 \right)<+\infty\\
  \label{eqn:integrable} \E\left( \sum_{m \in M} e^{\theta^* m} \left( \log\sum_{m \in M} e^{\theta^* m} \right)^2 \right) < +\infty.
\end{gather}
Then there exists $v_{k}$ such that
\begin{equation}
  \label{eqn:bg}
  \qquad\lim_{n \to +\infty} \frac{Y_n(1)}{n} = \lim_{n \to +\infty} \frac{Y_n(k)}{n} = v_k \quad \text{a.s.}
\end{equation}
and moreover $\lim_{k \to +\infty} (\log k)^2 (v_k - v) = - \frac{\pi^2(\theta^*\sigma)^2}{2}$.
\end{theorem}

It is a straightforward computation to note that $(\theta^* m - \kappa(\theta^*), m \in \mathcal{M})$ is a point process satisfying the assumptions of Theorem~1.1 in \cite{Mal15b}. Precisely, equation (1.3) there is verified as $\mathcal{M}$ contains at least one element a.s. and more than one element on average. Equation (1.4) comes from
\[
  \E\left( \sum_{m \in \mathcal{M}} e^{\theta^* m - \kappa(\theta^*)} \right) = \E\left( \sum_{m \in \mathcal{M}} e^{\theta^* m} \right) e^{- \kappa(\theta^*)}= 1.
\]
Moreover, the random variable $X$ whose law is defined by 
\[\P(X \leq x) = \E\left(\sum_{m \in M} \ind{\theta^* m - \kappa(\theta^*) \leq x} e^{\theta^* m - \kappa(\theta^*)}\right)\]
satisfies $\E(X^2) < \infty$, as by \eqref{eqn:finiteLT},
\begin{equation}
  \label{eqn:finiteVariance}
  \E(X^2) = \E\left(\sum_{m \in M} (\theta^* m - \kappa(\theta^*))^2 e^{\theta^* m - \kappa(\theta^*)}\right)= (\theta^*)^2 \kappa''(\theta^*) < \infty.
\end{equation}
Hence $X$ is in the domain of attraction of the normal distribution, so that $\alpha$ from \cite{Mal15b} is equal to $2$, $Y$ is the normal distribution and $(Y_t,t\geq0)$ is a standard Brownian motion. Thus the function $L^*$ defined in \cite{Mal15b} verifies $\lim_{x \to \infty} L^*(x) = (\theta^* \sigma)^2$ (note that \cite{Mal15b} contains a typo in formula (1.6), where $Y$ should be $X$) and the constant is $C_* = \frac{\pi^2}{2}$. Finally \eqref{eqn:boundary} immediately implies (1.10) in \cite{Mal15b}, and \eqref{eqn:integrable} together with \eqref{eqn:finiteVariance} implies (1.9) there. Hence the conclusions of \cite[Theorem 1.1]{Mal15b} hold.


We combine Lemma~\ref{lem:coupbrw} and Theorem~\ref{thm:Mallein} to bound the asymptotic behaviour of the speed of the continuous-time branching random walk with selection. We start with the upper bound.

\begin{lemma}
\label{lem:ubbrw}
Let $c_k$ be the speed of the continuous-time $k$-branching random walk defined in Lemma~\ref{lem:vitesseBRWSpecialCase}. Then we have
\[
\limsup_{k \to +\infty} (\log k)^2 (c_k - e) \leq - \frac{\pi^2 e}{2}.
\]
\end{lemma}

\begin{proof}
%
Let $(Y_t(u), u \in \calN_t)_{t \geq 0}$ be a continuous-time branching random walk without selection, in which particles create one child to their right at rate $1$ and starting with $k$ individuals at position $0$ at time $0$. Let $k \in \N$, we define the càdlàg adapted process $K_t$ as follows: at each integer time $n \in \N$, we set $K_n=k$ and for all $s \in [n,n+1]$, $K_s$ is the number of descendants at time $s$ of the $k$ individuals alive at time $n$ in $Y^K$. In other words, $Y^K$ is a continuous-time branching random walk with selection in which at each integer time, the rightmost $k$ particles are selected to survive. No additional killing of particles is made.

It appears clear that $K_t \geq k$ a.s. for all $t \geq 0$, therefore by Lemma~\ref{lem:coupbrw}, one can couple the branching random walks with selection $Y^K$ and $Y^k$ in such a way that $Y^K_t(1) \geq Y^k_t(1)$ a.s. As a result, we obtain
\begin{equation}
  \label{1234}
  c_k \leq \liminf_{t \to \infty} \frac{Y^K_t(1)}{t} \quad \text{a.s.}
\end{equation}

We now observe that $Y^K$ can also be constructed as a discrete-time branching random walk with selection. Indeed, each particle alive at time $n \in \N$ gives birth at time $n+1$ to a point process of individuals, distributed as $(\widehat{Y}_1(u), u \in \hat{\calN}_1)$, where $(\widehat{Y}_t(u), u \in \hat{\calN}_t)_{t \geq 0}$ is a continuous-time branching random walk without selection, in which particles create one child to their right at rate $1$ and starting with a single individual at position $0$ at time $0$. Then at time $n+1$, the rightmost $k$ particles are selected. We thus conclude that $(Y^K_n, n \geq 0)$ is a discrete-time $k$-branching random walk.

Let $\theta \in \mathbb{C}$, we compute for all $t \geq 0$, $f_t(\theta) = \E\left( \sum_{u \in \hat{\calN}_t} e^{\theta \widehat{Y}_t(u)} \right)$. As the first branching time of the process is exponentially distributed with parameter $1$, and after this reproduction event one particle at position $0$ and one particle at position $1$ start independent copies of the branching process from their position, we have
\begin{equation}
  \label{eqn:equadiff}
  f_t(\theta) = e^{-t}\left(1 + \int_0^t e^{s} f_s(\theta) (1 + e^\theta) ds\right) \text{ for all $t \geq 0$}.
\end{equation}
In particular, if $\theta \in i\R$, we have
\[
  |f_t(\theta)| \leq \E\left( \sum_{u \in \hat{\calN}_t} |e^{\theta \widehat{Y}_t(u)}| \right) \leq \E\left( \#\hat{\calN}_t \right) = e^t,
\]
where we used that each particle creates one child at rate $1$, so $\#\hat{\calN}_t$ has exponential distribution with parameter $e^{-t}$. Then, by the Cauchy-Lipschitz theorem applied to the linear differential equation \eqref{eqn:equadiff}, we conclude that $f_t= e^{te^\theta}$ for all $\theta \in i\R$. As a result, by analytic continuation, we deduce that $f_t(\theta) = e^{t e^{\theta}}$ for all $\theta \in \mathbb{C}$. This type of computation was first made in \cite{Uch}, we refer to \cite[Lemma 4.5]{BeM19} for a similar computation, as $\widehat{Y}$ can be thought of as a branching Lévy process with finite birth intensity.

As a result, we deduce that we have
\[
  \kappa(\theta) := \log \E\left( \sum_{u \in \hat{\calN}_1} e^{\theta \widehat{Y}_1(u)} \right) = e^\theta.
\]
From this, straightforward computations show that $\theta^* = 1$ and $v = \sigma^2= e$.

Moreover, \eqref{eqn:boundary} is verified: as the trajectories in $\widehat{Y}_t$ are non-decreasing, we have
\[
  \P\left( \max_{u \in \hat{\calN}_1} \widehat{Y}_1(u) < 0 \right) = 0
\]
and by the above computations and the Markov inequality, for all $y \geq 0$
\[
  \P\left( \max_{u \in \hat{\calN}_1} \widehat{Y}_1(u) > y \right) \leq \E\left( \sum_{u \in \hat{\calN}_1} e^{\widehat{Y}_1(u) - y} \right) \leq e^{e} e^{-y},
\]
proving that $\left| \max_{u \in \hat{\calN}_1} \widehat{Y}_1(u)\right|$ has exponential tails, hence a finite second moment.

We now show that \eqref{eqn:integrable} holds as well. Note there exists $C>0$ such that $x (\log x)^2 \leq C x^2+1$ for all $x > 0$, therefore for all $\theta > 0$,
\begin{multline*}
  \E\left( \sum_{u \in \hat{\calN}_1} e^{\theta \widehat{Y}_1(u)} \left( \log\sum_{u \in \hat{\calN}_1} e^{\theta \widehat{Y}_1(u)} \right)^2 \right) \\
  \leq C  \E\left( \left(\sum_{u \in \hat{\calN}_1} e^{\theta \widehat{Y}_1(u)} \right)^2 \right)+1 =: C g_t(\theta)+1.
\end{multline*}
With the same reasoning as for the computation of $f_t$, for all $\theta \in \mathbb{C}$ and $t \geq 0$, we have
\[
  g_t(\theta) = e^{-t} \left( 1 + \int_0^t e^s \left(g_s(\theta)(1 + e^{2\theta}) + 2 e^\theta f_s(\theta)^2\right) ds\right).
\]
This equation can be solved for $\theta \in i\R$ as $g_t(\theta) = \frac{e^\theta}{e^\theta - 2} e^{e^{2\theta} t} - \frac{2}{e^\theta - 2} e^{2 e^\theta t}$. Then, by analytic continuation, this formula also holds for all $\theta \in \R\backslash \{\log 2\}$ and is analytically continued by $g_t(\log 2) = (4 t + 1)e^{4t}$. So $g_1(\theta) < \infty$, completing the proof of \eqref{eqn:integrable}.

As a result, we can apply Theorem~\ref{thm:Mallein} and we obtain
\[
  \lim_{n \to \infty} \frac{Y^K_n(1)}{n} = v_k \quad \text{a.s.}
\]
with $v_k - e \sim - \frac{\pi^2e}{2 (\log k)^2}$ as $k \to \infty$. By \eqref{1234}, we have $v_k \geq c_k$, which completes the proof.
\end{proof}

The lower bound is obtained in a similar yet more involved fashion. The proof of this lemma is adapted from \cite[Section 4.4]{Mal15b}.
\begin{lemma}
\label{lem:lbbrw}
Let $c_k$ be the speed of the continuous-time $k$-branching random walk defined in Lemma~\ref{lem:vitesseBRWSpecialCase}. Then we have
\[
\liminf_{k \to +\infty} (\log k)^2 (c_k-e) \geq - \frac{\pi^2 e}{2}.
\]
\end{lemma}

\begin{proof}
In this proof, we construct a continuous-time particle process $\tilde{Y}$ that evolves similarly to a discrete-time branching random walk with selection, with frequent renovation events, and that can be coupled with the $k$-BRW $Y^k$ in such a way that its maximal displacement is smaller than the maximal displacement of $Y^k$. Given $a \in (0,1)$,  the process typically evolves like a discrete-time $\ceil{a k}$-branching random walk, and on a time scale of order $(\log k)^3$, every particle in the process is killed and replaced by $\ceil{a k}$ particles starting from the smallest position in $\tilde{Y}$ at that time.

Let $a \in (0,1)$, we set $p = \ceil{a k}$. Let $(Y_t(u), u \in \calN_t)$ (resp. $(\widehat{Y}_t(u), u \in \hat{\calN}_t)$) be a continuous-time branching random walk, starting from $k$ particles (resp. a single particle) located at position 0. As $\lim_{t \to 0} \E(\#\hat{\calN}_t) + t =1 < \frac{1}{a}$, there exists $0 < \beta < 1$ such that $\E(\#\calN_\beta) < \frac{1}{a} - \beta$. We introduce the point process $M^\beta := (\hat{Y}_\beta(u), u \in \hat{\calN}_\beta)$.

Let $(Z^p_n(j), j \leq p)_n$ be a discrete-time branching random walk with selection of the rightmost $p$ particles, with reproduction law $M^\beta$, starting with $p$ particles located at position 0 at time $0$. With the same computations as in the proof of Lemma~\ref{lem:ubbrw}, we obtain $\kappa(\theta) = \beta e^\theta$ for every $\theta>0$ and
\[
  v = \beta e, \quad \theta^* = 1 \quad \text{and} \quad \sigma^2 = \beta e.
\]
Let $\eta>0$ and $\chi_p = \beta \frac{\pi^2 e}{2 (\log p)^2}$. Applying \cite[Lemma 4.6]{Mal15b}, there exists $\gamma>0$ such that for all $p \geq 1$ large enough, we have
\begin{equation}
  \label{eqn:badEvent2}
  \P\left( \forall n \leq (\log p)^3, Z^p_n(p) - n\beta e \leq -n (1+\eta) \chi_p \right) \leq \exp(-p^\gamma).
\end{equation}
To translate \cite[Lemma 4.6]{Mal15b} into \eqref{eqn:badEvent2}, one should specialize the quantities defined in \cite[Lemma 4.6]{Mal15b} similarly to what is done below Theorem~\ref{thm:Mallein}, setting $N=p$, $\calL$ the law of $\theta^* M^\beta - \kappa(\theta^*)$, $\lambda=1$, $\epsilon=\eta$, $\delta=\gamma$, $\alpha=2$, and $\nu_N=\chi_p/\beta e$, recalling that $L^*(x) \to \beta e$ as $x \to \infty$. Lemma 4.6 in \cite{Mal15b} is associated to the branching random walk $X^N_n(N)=Z^p_n(p) - n\beta e$.

Note that $\nu_N$ is defined in~\cite{Mal15b} by formula (4.1) and is, up to a sign, nothing but the right-hand side of the formula in \cite[Theorem 1.1]{Mal15b}. The assumptions required for \cite[Lemma 4.6]{Mal15b} are the same as those of \cite[Theorem 1.1]{Mal15b} and one checks that they are satisfied exactly as in the proof of Lemma~\ref{lem:ubbrw} since $M^\beta$ satisfies the assumptions appearing in and before the statement of Theorem~\ref{thm:Mallein}. The estimate \eqref{eqn:badEvent2} will be used later in the proof.

We observe that, as in the proof of Lemma~\ref{lem:ubbrw}, $(Z^p_n(j), j \leq p)_n$ can also be constructed as the values taken at discrete times by a continuous-time $P$-BRW, for a given adapted integer-valued càdlàg process $P$. More precisely, we introduce $(P_t)$ defined by $P_{n\beta} = p$ for any $n \geq 0$ and for any $t \in (n\beta,(n+1)\beta)$, $P_t$ is the number of descendants at time $t$ of particles alive at time $n \beta$. We have
\[
  \left(\left( Y^P_{n\beta}(j), j \leq p\right), n \geq 0 \right) \egaldistr \left(\left( Z^p_n(j), j \leq p\right), n \geq 0 \right).
\]
For any $n \in \N$, we introduce the event $\calA^k_n = \left\{ \max_{t \leq \beta n} P_t \leq k \right\}$. By Lemma~\ref{lem:coupbrw}, we can couple $Y^k$ and $Y^P$ in such a way that
\[
  \forall x \in \R, \, \# \{ j \leq P : Y^P_{n\beta}(j) \geq x \} \leq \# \{ j \leq k : Y^k_{n\beta}(j) \geq x \} \text{ a.s. on }\calA^k_n.
\]

We bound from below the probability for $\calA^k_n$ to occur. As every particle makes at least one child, the process $P$ is non-decreasing on each interval $(n\beta,(n+1)\beta)$. Moreover, observe that $P_{\beta-}$ is the sum of $p$ i.i.d. random variables, each with the same distribution as the number $\#\hat{\calN}_\beta$ of particles alive at time $\beta$ in the continuous-time branching random walk $(\widehat{Y}_t)$. As $\#\hat{\calN}_\beta$ is a geometric random variable with parameter $e^{-\beta}$ (see \cite[p. 109]{AtN}), by construction of $\beta$ this random variable has mean smaller than $1/a$ and has some exponential moments.
By Cramér's large deviations theorem, there exists $\rho<1$ independent of k such that $\P(P_{\beta-} > k) < \rho^k$.
Therefore
\begin{equation}
  \label{eqn:badEvent1}
   \P( {\calA^k_n}^c ) \leq \sum_{j=0}^{n-1} \P( P_{j\beta-} > k) \leq n \rho^k.
\end{equation}

We now construct a continuous-time particle process $\tilde{Y}$, based on the $P$-BRW $Y^P$ that bounds from below the $k$-BRW $Y^k$. Let $n_p = (\log p)^3$, we set $T_0 = 0$. For any $t \geq 0$, we write $\tilde{Y}_t(1), \ldots$ the positions of the particles in $\tilde{Y}$ at time $t$, ranked in a non-increasing order, and $\tilde{m}_t$ the position of the leftmost particle at time $t$. The particle process $\tilde{Y}$ starts at time $0$ with $p$ particles at position $0$ and behaves like $Y^P$ until the waiting time
\begin{align*}
  &T_1 = \min( \beta n_p, T_1^{(1)}, T_1^{(2)} ), \quad \text{where } T_1^{(1)} = \inf\left\{ t \geq 0 : P_t \geq k \right\}\\
 \text{and } &T_1^{(2)} = \beta \inf\left\{ n \in \N : \tilde{m}_{n\beta} > n(\beta e - \chi_p(1 + \eta)) \right\}.
\end{align*}

At time $T_1$, every particle in $\tilde{Y}$ is killed and $p$ new particles are positioned at $\tilde{m}_{T_1-}$ if $P_{T_1} > k$ (i.e. $T_1 = T_1^{(1)}$) and at position $\tilde{m}_{T_1}$ otherwise. By the above coupling between $Y^k$ and $Y^P$, in both cases there are at time $T_1$ at least $p$ particles in $Y^k$ to the right of the $p$ newborn particles in $\tilde{Y}$.

Let $\ell \in \N$, we assume the process $\tilde{Y}$ has been constructed until time $T_\ell$. After this time, it evolves as a $P$-BRW until time
\begin{align*}
  &T_{\ell+1} = \min( T_\ell + \beta n_p, T_{\ell+1}^{(1)}, T_{\ell+1}^{(2)} ), \quad \text{where } T_{\ell+1}^{(1)} = \inf\left\{ t \geq T_\ell : P_t \geq k \right\}\\
 \text{and } &T_{\ell+1}^{(2)} = T_\ell + \beta\inf\left\{ n \in \N : \tilde{m}_{T_\ell+\beta n} - \tilde{m}_{T_\ell} > n(\beta e - \chi_p(1 + \eta)) \right\}.
\end{align*}
At time $T_{\ell+1}$, every particle in $\tilde{Y}$ is killed and $p$ new particles are positioned at $\tilde{m}_{T_{\ell+1}-}$ if $P_{T_{\ell+1}} > k$ (i.e. $T_{\ell+1} = T_{\ell+1}^{(1)}$) and at position $\tilde{m}_{T_{\ell+1}}$ otherwise.

By induction and the construction of the process, we observe that $\tilde{Y}$ can be coupled with $Y^k$ in such a way that for any $t \geq 0$, we have
\[
  \forall x \in \R, \, \# \{ j \leq P_t : \tilde{Y}_t(j) \geq x \} \leq \# \{ j \leq k : Y^k_t(j) \geq x \}.
\]
As $\tilde{Y}_t(P_t) \leq Y^k_t(1)$ for any $t>0$, we obtain $\displaystyle\limsup_{t \to +\infty} t^{-1}(\tilde{m}_t-te) \leq c_k-e$, using Lemma~\ref{lem:existSpeed}.

Moreover, observe that $(T_{\ell+1}-T_\ell)_\ell$ and $(\tilde{m}_{T_{\ell+1}}-\tilde{m}_{T_\ell})_\ell$ are i.i.d. sequences of random variables. Consequently, by the law of large numbers we have
\[
  \lim_{\ell \to \infty} \frac{T_\ell}{\ell} = \E(T_1) \quad \text{ and } \quad \lim_{\ell \to \infty} \frac{\tilde{m}_{T_{\ell}}}{\ell} = \E(\tilde{m}_{T_1}) \quad \text{a.s.},
\]
where $\E(T_1) \leq \beta n_p < \infty$ by definition, and $ \tilde{m}_{T_1} \geq 0$ a.s. Therefore, we have
\[\lim_{\ell \to \infty} \frac{\tilde{m}_{T_\ell} - T_\ell e}{T_\ell} = \frac{\E(\tilde{m}_{T_1}-T_1 e)}{\E(T_1)} \leq c_k - e.\]

As a result, to conclude the proof it is enough to bound $\E(\tilde{m}_{T_1}-T_1 e)$ from below. We introduce the event $G = \{T_1 = T^{(2)}_1 <T^{(1)}_1\}$. By definition of $T_1$,
\begin{equation}
  \label{eqn:presquelafin}
  \E(\tilde{m}_{T_1}-T_1 e)\geq
  \E\left( - \tfrac{T_1}{\beta} \chi_p(1 + \eta) \indset{G} \right) +\E\left( (\tilde{m}_{T_1-}-T_1 e) \indset{G^c} \right).
\end{equation}
Observe that until time $T_1-$, $\tilde{Y}$ behaves as a $P$-BRW. In particular, the trajectories of particles are non-decreasing, therefore
\begin{align*}
  \E\left( (\tilde{m}_{T_1-}-T_1 e) \indset{G^c} \right) &\geq - \E(T_1 e \indset{G^c})\\
   &\geq -e \beta n_p \left( e^{-p^\gamma} + n_p\rho^k \right) = o\left(\left(\log k\right)^{-4}\right),
\end{align*}
by \eqref{eqn:badEvent2} and \eqref{eqn:badEvent1}.

As a consequence, \eqref{eqn:presquelafin} yields
\[
  \liminf_{k \to +\infty} (\log k)^2 (c_k - e) \geq \liminf_{k \to +\infty} -(\log k)^2 \frac{\chi_p}{\beta}(1 + \eta) \frac{\E(T_1 \indset{G})}{\E(T_1)}
  \geq -\frac{\pi^2 e}{2}(1+\eta).
\]
We conclude the proof by letting $\eta \to 0$.
\end{proof}

The last statement of Lemma~\ref{lem:vitesseBRWSpecialCase} is a combination of Lemmas~\ref{lem:ubbrw} and~\ref{lem:lbbrw}.

\paragraph*{Acknowledgements}

We would like to thank Ksenia Chernysh, Sergey Foss, Patricia Hersh, Richard Kenyon, Takis Konstantopoulos and Jean-Fran\c{c}ois Rupprecht for fruitful discussions and Persi Diaconis for pointing out the reference~\cite{AP}.

\label{Bibliography}
\bibliographystyle{plain}
\bibliography{bibliographie}

\Addresses

\end{document}